\documentclass[11pt,english,"authoryear"]{article}
\usepackage[T1]{fontenc}
\usepackage[latin9]{inputenc}
\usepackage{geometry}
\usepackage{url}
\usepackage{setspace}
\usepackage{float}
\usepackage{amsmath}
\usepackage{amsthm}
\usepackage{amssymb}
\usepackage{adjustbox}
\usepackage{multirow}
\usepackage[authoryear]{natbib}
\usepackage{array}
\usepackage{makecell}
\usepackage[super]{nth}
\usepackage{longtable}
\usepackage{comment}
\usepackage{lscape}

\usepackage{mathtools}
\DeclarePairedDelimiter\ceil{\lceil}{\rceil}
\DeclarePairedDelimiter\floor{\lfloor}{\rfloor}

\makeatletter

\floatstyle{ruled}
\newfloat{algorithm}{tbp}{loa}
\providecommand{\algorithmname}{Algorithm}
\floatname{algorithm}{\protect\algorithmname}

\usepackage{pdfpages}
\usepackage{changepage}
\usepackage{tikz}
\usepackage{pgfplots}
\usepackage{amsmath}
\usepackage{geometry}
\usepackage{algorithm,algorithmic}
\usepackage[affil-it]{authblk}

\makeatother

\usepackage{babel}
\newtheorem{theorem}{Theorem}
\newtheorem{lemma}{Lemma}
\newtheorem*{corollary}{Corollary}
\newtheorem{corollar}{Corollary}
\newtheorem{proposition}{Proposition}
\newtheorem*{propos}{Proposition}

\begin{document}

\title{Consistency Cuts for Dantzig-Wolfe Reformulations}
\author[1]{Jens Vinther Clausen \thanks{corresponding author\newline \emph{Email addresses:} jevcl@dtu.dk (Jens Vinther Clausen), rmlu@dtu.dk (Richard Lusby), ropke@dtu.dk (Stefan Ropke)}}
\author[1]{Richard Lusby}
\author[1]{Stefan Ropke}

\affil[1]{\scriptsize{Department of Technology, Management and Economics, Technical University of Denmark, Denmark}}

\maketitle

\begin{abstract}

\noindent This paper introduces a family of valid inequalities, that we term consistency cuts, to be applied to a Dantzig-Wolfe reformulation (or decomposition) with linking variables. We prove that these cuts ensure an integer solution to the corresponding Dantzig-Wolfe relaxation when certain criteria to the structure of the decomposition are met.
We implement the cuts and use them to solve a commonly used test set of 200 instances of the temporal knapsack problem. We assess the performance with and without the cuts and compare further to CPLEX and other solution methods that have historically been used to solve the test set. By separating consistency cuts we show that we can obtain optimal integer solutions much faster than the other methods and even solve the remaining unsolved problems in the test set. We also perform a second test on instances from the MIPLIB 2017 online library of mixed-integer programs, showing the potential of the cuts on a wider range of problems.
\end{abstract}

\noindent \emph{Keywords:} Integer . < Programming, Cutting plane/facet < Algorithms < Integer . < Programming, Theory < Integer . < Programming

\clearpage

\section{Introduction}

\emph{Dantzig-Wolfe reformulation} (DW reformulation) or \emph{DW decomposition} was originally proposed in \citep{dantzig1960decomposition} as a way to reformulate a \emph{linear program} (LP) in order to exploit a specific problem structure. It has since been applied to \emph{integer linear programs} (IPs), and \emph{mixed-integer linear programs} (MIPs) \citep{savelsbergh1997branch,barnhart1998branch,vanderbeck2010reformulation}. Consider a general IP of the form:

\begin{singlespace}
$\,$

\noindent($\mathsf{IP}$) \hspace{1.9in} $\mathrm{min}\{cx|Ax\leq b,x\in\mathbb{Z}^{n}\}$

$\,$
\end{singlespace}

\noindent where $x$ is a vector of decision variables, $A$ is the constraint matrix, $b$ is the right-hand side vector of the constraints, and $c$ is the cost vector. Let $\mathcal{S}$ denote the set of constraints. In general, DW reformulation splits $\mathcal{S}$ into $k+1$ disjoint sets (blocks) $\mathcal{S}_{i}$ for $i\in\{0,\ldots,k\}$ such that $\bigcup_{i\in\{0,\ldots,k\}}\mathcal{S}_{i}=\mathcal{S}$. We refer to the choice of what constraints go into which blocks as the \emph{decomposition}.

In most applications of DW reformulation the decomposition is chosen such that the corresponding constraint matrix $A$ has \emph{single-bordered block-diagonal form}:

\begin{singlespace}
\begin{equation*}
\begin{aligned}A= & \left[\begin{array}{ccccc}
H & E_{1} & E_{2} & \cdots & E_{k}\\
 & D_{1}\\
 &  & D_{2}\\
 &  &  & \ddots\\
 &  &  &  & D_{k}
\end{array}\right]\end{aligned}
\end{equation*}
\end{singlespace}

\noindent where $D_{1}$,$D_{2}$,\ldots,$D_{k}$ are the coefficients of the variables in the constraints of sets $\mathcal{S}_{1}$,$\mathcal{S}_{2}$,\ldots,$\mathcal{S}_{k}$ respectively, and $H,E_{1},E_{2},\ldots,E_{k}$
are the coefficients of the variables in the constraints of $\mathcal{S}_{0}$, which we refer to as the \emph{coupling constraints}. This structure implies that the variables that appear in the constraints of each of the $\mathcal{S}_{i}$ for $i\in\{1,\ldots,k\}$ are unique to that set and $\mathcal{S}_{0}$.

If the constraint matrix does not have single-bordered block-diagonal form given the chosen decomposition, it means that there are one or more variables that have nonzero coefficients in the constraints of more than one of the sets $\mathcal{S}_{i}$
for $i\in\{1,\ldots,k\}$. We refer to these variables as \emph{linking variables}. The corresponding constraint matrix has \emph{double-bordered block-diagonal form}:

\begin{singlespace}
\begin{equation}
\begin{aligned}A= & \left[\begin{array}{ccccc}
H & E_{1} & E_{2} & \cdots & E_{k}\\
F_{1} & D_{1}\\
F_{2} &  & D_{2}\\
\vdots &  &  & \ddots\\
F_{k} &  &  &  & D_{k}
\end{array}\right]\end{aligned}
\label{eq:arrow_head_struct}
\end{equation}
\end{singlespace}

\noindent where $H,F_{1}$, $F_{2}$, \ldots, $F_{k}$ are the coefficients of the linking variables. Examples of DW reformulations based on double-bordered block-diagonal form can be found in \citep{bergner2015automatic, caprara2013uncommon, mingozzi2017exact, gschwind2017stabilized, zhao2018fixed}. It is well known that \emph{Lagrangian relaxation} and DW reformulation are equivalent in many ways \citep{guignard2003lagrangean, vanderbeck2006generic, pessoa2018automation}. Decompositions based on double-bordered block-diagonal matrices are also used within research based on Lagrangian relaxation. Here the technique is known as \emph{Lagrangian decomposition} \citep{guignard1987lagrangean, guignard2003lagrangean}, and appears to have been used earlier and more frequently compared to its DW counterpart. Examples are \citep{glover1988layering, nilsson1997variable, wu2003decomposition,freville2005multidimensional, aguado2009fixed, tang2011improved}.

The relaxation of the DW reformulation (\emph{the DW relaxation}) often includes a huge number of variables and is therefore solved using \emph{column generation} \citep{lubbecke2005selected}. This involves iteratively generating columns for the DW relaxation (\emph{master problem}) by solving a set of \emph{subproblems}, one for each block in the decomposition. The DW relaxation bounds can be used in a \emph{branch-and-price} algorithm \citep{barnhart1998branch} to find integer solutions. Ideally, the DW relaxation provides a tighter bound than the LP relaxation of the original model, without being too much harder to compute. This potentially means that less branching will be necessary and an optimal integer solution would be found faster than if using a branch-and-bound algorithm \citep{lawler1966branch} using the LP relaxation of the original model.

In the absence of coupling constraints, matrix $A$ has a structure that is also amenable to Benders Decomposition, see e.g.~\citep{benders1962}; on fixing the values of the linking variables, the constraint coefficients of which are contained in matrices $F_1, F_2,\dots, F_k$,  the resulting constraint matrix can be decomposed into $k$ distinct components (subproblems), that can be solved separately. The constraint matrix of subproblem $i$ is given by $D_i$. In Benders decomposition, a master problem is used to optimize the values of the linking variables, while the subproblems are used to verify feasibility, and optimality, of the solution to the master problem.  Benders decomposition was originally designed to work in situations where the subproblems are continuous linear programs since duality theory is used to derive feasibility and optimality cuts, which are used to guide the master problem search. The technique can also be applied in situations where the subproblems define integer programs, see e.g. \citep{hooker2003a}; however, the lack of duality makes the method much less generic. The derivation of feasibility and optimality cuts is often tailored to the specific application, utilizing the structure of the problem.

The \textit{temporal knapsack problem} (TKP) is a generalization of the well-known \textit{knapsack problem} (KP) \citep{kellerer2003knapsack}. The TKP was named in \citep{bartlett2005temporal} which lists a number of applications regarding management of sparse resources: CPU time, communication bandwidth, computer memory, or disc space. It was later encountered in railway service design in \citep{caprara2011freight}. Multiple solution methods have been proposed for the TKP: \citep{caprara2013uncommon, gschwind2017stabilized, caprara2016solving} all use DW reformulation whereas \citep{clautiaux2019dynamic} present a dynamic programming method. \citep{caprara2013uncommon} compiled a benchmark set of 200 instances of the TKP which has since become the standard test set in the literature when benchmarking solution methods for the TKP.

\emph{Valid linear inequalities} (cuts) are inequalities that can be added to an LP in order to reduce the solution space and tighten the relaxation bound. When incorporated in a branch-and-price framework, we refer to the complete algorithm as \emph{branch-and-cut-and-price} \citep{Belov2005, de2003integer}; ideally, the cuts reduce the amount of branching needed. We differentiate between two types of cuts: robust cuts, e.g. \citep{kohl19992, van2000time, fukasawa2006robust, desaulniers2011cutting}, are cuts that only modify the master problem, and non-robust cuts, e.g. \citep{nemhauser1991polyhedral, jepsen2008subset, petersen2008chvatal, spoorendonk2010clique, dabia2019cover, desaulniers2011cutting}, are cuts that modify both the master and subproblems.

This paper has four main contributions: 1) We introduce a set of non-robust cuts that we term \emph{consistency cuts} to be applied to DW relaxations with binary linking variables. 2) We prove that these cuts ensure optimal solutions to the original problem as long as the structure of the decomposition satisfies certain conditions. For instance, these conditions are satisfied for the TKP, when using the decomposition proposed by \citep{caprara2013uncommon}. 3) We use this decomposition in a branch-and-cut-and-price algorithm \citep{vanderbeck2010reformulation} to solve the 200 instances in the TKP benchmark set. The algorithm finds optimal integer solutions much faster than CPLEX and other solution methods proposed in the literature. It is the first algorithm to solve all instances in the benchmark set. 4) We apply the cuts to 52 instances from the MIPLIB2017 online library of MIPs \citep{MIPLIB2017}. For each instance 5 different decompositions are tested, and for 20 of the instances the cuts improve the root bound of at least one of the 5 decompositions. Even though these instances are easily solved to optimality with CPLEX, this demonstrates the potential of the cuts on generic MIPs.


The rest of the paper is structured as follows. Section \ref{DWR} shows how to apply DW reformulation to the general problem ($\mathsf{IP}$). In Section \ref{sharingCuts} the consistency cuts are presented and their theoretical properties are examined. Section \ref{TKP} describes the TKP and its decomposition. Section \ref{Implementation} presents the corresponding master problem and subproblems for column generation. It also elaborates on the implementation details of the algorithm. In Section \ref{CompResults}, computational tests are presented. Finally, we conclude the paper and present ideas for future research directions in Section \ref{Conclusion}.

\section{Dantzig-Wolfe reformulation with linking variables}\label{DWR}

A detailed description of DW reformulation without linking variables can be found in \citep{vanderbeck2010reformulation}. Very few descriptions of DW reformulation with linking variables for generic problems exist \citep{bergner2015automatic}. Therefore, in the following, we present DW reformulation with linking variables for the general problem ($\mathsf{IP}$) using \emph{convexification} as opposed to \emph{discretization}, see \citep{vanderbeck2010reformulation}. This definition naturally extends to MIPs, but for ease of presentation we use an IP here. The constraint matrix $A$ in ($\mathsf{IP}$) is assumed to have double-bordered block-diagonal form, i.e. can be written as (\ref{eq:arrow_head_struct}). We define, for each $i\in\{1,\ldots,k\}$, $x_{i}$ as the vector of variables that are unique to the constraints of $\mathcal{S}_{i}$, $n_{i}$ as the length of $x_{i}$, $b_{i}$ as the right-hand side vector of the constraints in $\mathcal{S}_{i}$, and $c_{i}$ as the cost vector of $x_{i}$. Furthermore, we define $x_{0}$ as the vector of linking variables, $c_{0}$ as the cost vector of $x_{0}$, and $b_{0}$ as the right-hand side vector of the constraints in $\mathcal{S}_{0}$. The corresponding IP is:
\begin{singlespace}
$\,$

\begin{minipage}{\textwidth-\parindent}

\begin{equation}
\mathrm{min}\:cx\label{eq:VS1_obj}
\end{equation}

subject to

\begin{alignat}{2}
Hx_{0}+\sum_{i=1}^{k}E_{i}x_{i} & \leq b_{0}\\
F_{i}x_{0}+D_{i}x_{i} & \leq b_{i} & \quad & \forall i\in\{1,\ldots,k\}\label{eq:DW1_blocks}\\
x_{i} & \in\mathbb{Z}^{n_{i}} & \quad & \forall i\in\{0,\ldots,k\}
\end{alignat}

\end{minipage}

$\,$
\end{singlespace}

\noindent Throughout this paper we assume that the decomposition is chosen such that 

\begin{singlespace}
\[
\bar{X}_{i}={\textstyle \left\{ \left(\begin{array}{c}
x_{0}\\
x_{i}
\end{array}\right)\in\mathbb{\mathbb{R}}^{n_{0}+n_{i}}\left|F_{i}x_{0}+D_{i}x_{i}\leq b_{i}\right.\right\} }
\]
\end{singlespace}

\noindent for $i\in\{1,\ldots,k\}$ are bounded. This eliminates the need for extreme rays when applying Minkowski-Weyl's Theorem \citep{schrijver1986theory} later. Next, we convexify the constraints of each block separately, i.e. we define:

\begin{singlespace}
\begin{equation}
X_{i}={\textstyle \operatorname{conv}}\left\{ \left(\begin{array}{c}
x_{0}\\
x_{i}
\end{array}\right)\in\mathbb{Z}^{n_{0}+n_{i}}\left|\left(\begin{array}{c}
x_{0}\\
x_{i}
\end{array}\right)\in\bar{X}_{i}\right.\right\}\label{eq:convexhull}
\end{equation}
\end{singlespace}

\noindent for $i\in\{1,\ldots,k\}$, i.e. the convex hulls of the integer feasible points to the constraints in each of the sets $\mathcal{S}_{i}$ for $i\in\{1,\ldots,k\}$. Since $\bar{X}_{i}$ is bounded so is $X_{i}$. By Minkowski-Weyl's Theorem we have that a bounded convex set can be represented as a convex combination of its extreme points. 

We define, for each $i\in\{1,\ldots,k\}$, $(\begin{array}{c c c}\bar{x}_{0p}^{i} & \bar{x}_{ip}\end{array})^\intercal$ for $p\in\mathcal{P}_{i}$ as the extreme points of $X_{i}$, i.e. $\bar{x}_{0p}^{i}$ is the value of the linking variables and $\bar{x}_{ip}$ is the value of the variables that only appear in the constraints of $\mathcal{S}_{i}$ and $\mathcal{S}_{0}$. Notice that we need to add an extra $i$-index to $\bar{x}_{0p}^{i}$ to be able to identify from which $X_{i}$ it originates. We also introduce variables $\lambda_{ip}$ as the weight with which extreme point $p$ of block $i$ is used. This allows us to write:

\begin{singlespace}
\[
X'_{i}=\left\{ \left(\begin{array}{c}
x_{0}\\
x_{i}
\end{array}\right)\in\mathbb{R}^{n_{0}+n_{i}}\left|\left(\begin{array}{c}
x_{0}\\
x_{i}
\end{array}\right)=\sum_{p\in\mathcal{P}_{i}}\lambda_{ip}\left(\begin{array}{c}
\bar{x}_{0p}^{i}\\
\bar{x}_{ip}
\end{array}\right),\sum_{p\in\mathcal{P}_{i}}\lambda_{ip}=1,\lambda_{ip}\geq0\ \forall p\in\mathcal{P}_{i}\right.\right\}
\]
\end{singlespace}

\noindent According to Minkowski-Weyl's Theorem, $X_{i}$ and $X'_{i}$ are equivalent. We now substitute (\ref{eq:DW1_blocks}) with $(\begin{array}{c c c}
x_{0} & x_{i}
\end{array})^\intercal\in X'_{i}\ \forall i\in\{1,\ldots,k\}$ which leads to the model:

\begin{singlespace}

$\,$

\noindent\begin{minipage}{\textwidth}

\noindent($\mathsf{DW1}$)

\begin{minipage}{\textwidth-\parindent}

\begin{equation}
\mathrm{min}\:c_{0}x_{0}+\sum_{i=1}^{k}c_{i}x_{i}\label{eq:VS2_obj}
\end{equation}

subject to

\begin{alignat}{2}
Hx_{0}+\sum_{i=1}^{k}E_{i}x_{i} & \leq b_{0}\label{eq:VS2_set1}\\
\left(\begin{array}{c}
x_{0}\\
x_{i}
\end{array}\right) & =\sum_{p\in\mathcal{P}_{i}}\lambda_{ip}\left(\begin{array}{c}
\bar{x}_{0p}^{i}\\
\bar{x}_{ip}
\end{array}\right) & \quad & \forall i\in\{1,\ldots,k\}\label{eq:VS2_set2}\\
\sum_{p\in\mathcal{P}_{i}}\lambda_{ip} & =1 & \quad & \forall i\in\{1,\ldots,k\}\label{eq:VS2_conv1}\\
\lambda_{ip} & \geq0 & \quad & \forall p\in\mathcal{P}_{i},i\in\{1,\ldots,k\}\label{VS2_conv2}\\
x_{i} & \in\mathbb{Z}^{n_{i}} & \quad & \forall i\in\{0,\ldots,k\}\label{eq:VS2_dom}
\end{alignat}

\end{minipage}

\end{minipage}

$\,$

\end{singlespace}

\noindent Next, we can substitute the right-hand side of (\ref{eq:VS2_set2}) into the rest of the model. This is easily done for the $x_{i}$ part of the vector but needs a bit more thought for the $x_{0}$ part. Notice that $x_{0}$ has a unique right-hand side for each $i\in\{1,\ldots,k\}$ in (\ref{eq:VS2_set2}). We can pick any one of these right-hand sides and substitute it into the others, and the rest of the model. Naturally, we pick the first one. This substitution leads us to ($\mathsf{DW2}$), which we refer to as the DW reformulation:

\begin{singlespace}
$\,$

\noindent($\mathsf{DW2}$)

\begin{minipage}{\textwidth-\parindent}

\begin{equation}
\mathrm{min}\:c_{0}\sum_{p\in\mathcal{P}_{1}}\lambda_{1p}\bar{x}_{0p}^{1}+\sum_{i=1}^{k}c_{i}\sum_{p\in\mathcal{P}_{i}}\lambda_{ip}\bar{x}_{ip}\label{eq:VS3_obj}
\end{equation}

subject to

\begin{alignat}{2}
H\sum_{p\in\mathcal{P}_{1}}\lambda_{1p}\bar{x}_{0p}^{1}+\sum_{i=1}^{k}E_{i}\sum_{p\in\mathcal{P}_{i}}\lambda_{ip}\bar{x}_{ip} & \leq b_{0}\label{eq:VS3_set1}\\
\sum_{p\in\mathcal{P}_{1}}\lambda_{1p}\bar{x}_{0p}^{1}-\sum_{p\in\mathcal{P}_{i}}\lambda_{ip}\bar{x}_{0p}^{i} & =0 & \quad & \forall i\in\{2,\ldots,k\}\label{eq:VS3_split}\\
\sum_{p\in\mathcal{P}_{i}}\lambda_{ip} & =1 & \quad & \forall i\in\{1,\ldots,k\}\label{eq:VS3_conv1}\\
\lambda_{ip} & \geq0 & \quad & \forall p\in\mathcal{P}_{i},i\in\{1,\ldots,k\}\label{eq:VS3_conv2}\\
\sum_{p\in\mathcal{P}_{1}}\lambda_{1p}\bar{x}_{0p}^{1} & \in\mathbb{Z}^{n_{0}}\label{eq:VS3_dom1}\\
\sum_{p\in\mathcal{P}_{i}}\lambda_{ip}\bar{x}_{ip} & \in\mathbb{Z}^{n_{i}} & \quad & \forall i\in\{1,\ldots,k\}\label{eq:VS3_dom2}
\end{alignat}

\end{minipage}

$\,$
\end{singlespace}

\noindent Whereas the $x_{i}$ part of (\ref{eq:VS2_set2}) disappears when substituting due to its uniqueness, we are left with $k-1$ constraints from the $x_{0}$ part of it, see (\ref{eq:VS3_split}). We refer to these constraints as the \emph{variable linking} constraints. If the constraint matrix had single-bordered block-diagonal form, i.e. no linking variables, the DW reformulation would not have any variable linking constraints. In practice, we often use a decomposition where some of the linking variables only appear in the constraints of some of the sets $\mathcal{S}_{i}$ for $i\in\{1,\ldots,k\}$. If this is the case, we end up having fewer variable linking constraints. Of course, we could also have variables that only appear in the constraints of $\mathcal{S}_{0}$. Such variables are simply left unchanged with no substitution to be made. The LP relaxation of ($\mathsf{DW2}$) is given by (\ref{eq:VS3_obj})-(\ref{eq:VS3_conv2}), which we refer to as the DW relaxation or ($\mathsf{DW3}$).

\subsection{Example} \label{example1}

Consider the simple integer program:
\begin{singlespace}
\begin{equation}
\mathrm{max}\:3x_1+2x_2+2x_3+3x_4\label{eq:Ex_obj}
\end{equation}

subject to


\begin{alignat}{3}
2x_1\, + \, & x_2+x_3 & \ \leq\  & 2\label{eq:Ex_con1}\\
 & x_2+x_3+x_4 & \ \leq\  & 2\label{eq:Ex_con2}\\
 & x_1,x_2,x_3,x_4 & \ \in\  & \{0,1\}\label{eq:Ex_dom}
\end{alignat}
\end{singlespace}

\noindent We choose to decompose the problem by defining the three blocks $\mathcal{S}_{0}=\emptyset$, $\mathcal{S}_{1}=\{(\ref{eq:Ex_con1})\}$, and $\mathcal{S}_{2}=\{(\ref{eq:Ex_con2})\}$, i.e. $x_2$ and $x_3$ are linking variables. Table~\ref{tab:smallexample} shows the complete set of extreme points in the DW relaxation of model (\ref{eq:Ex_obj})-(\ref{eq:Ex_dom}). Observe that the linking variables do not contribute to the profit of the columns for the second block, see~\eqref{eq:VS3_obj}. The optimal solution uses four of the columns:

\begin{singlespace}
\begin{table}[H]
\centering
\begin{tabular}{c|ccccc|ccccccc}
\hline
 & \multicolumn{5}{c|}{Block 1} & \multicolumn{7}{c}{Block 2} \\
\hline
$\lambda$ & 0 & 0.5 & 0 & 0 & 0.5 & 0 & 0 & 0 & 0 & 0 & 0.5 & 0.5 \\
$c$ & 0 & 3 & 2 & 2 & 4 & 0 & 0 & 0 & 3 & 0 & 3 & 3 \\
\hline
$\bar{x}_1$ & 0 & 1 & 0 & 0 & 0 & - & - & - & - & - & - & - \\
$\bar{x}_2$ & 0 & 0 & 1 & 0 & 1 & 0 & 1 & 0 & 0 & 1 & 1 & 0 \\
$\bar{x}_3$ & 0 & 0 & 0 & 1 & 1 & 0 & 0 & 1 & 0 & 1 & 0 & 1 \\
$\bar{x}_4$ & - & - & - & - & - & 0 & 0 & 0 & 1 & 0 & 1 & 1 \\
\hline
\end{tabular}
\caption{A small example}
\label{tab:smallexample}
\end{table}
\end{singlespace}

\noindent Table~\ref{tab:smallexample} shows the columns for block~1 and block~2 separately. Notice that the columns from block~1 are missing an entry for $\bar{x}_4$. This is because $x_4$ is not part of~\eqref{eq:Ex_con1}. Likewise, the columns for block~2 are missing entries for $\bar{x}_1$. The second row indicates the weight of the column in the solution. Only the columns with weight 0.5 are used in this solution. The second row indicates the profit of each column. The weighted sum of the profits of this solution is 6.5. We know from~\eqref{eq:VS3_split} that in any feasible solution the weighted sum of each linking variable across all columns for each value of $i$ must be the same. Checking this, we see that both $\bar{x}_2$ and $\bar{x}_3$ have a value of 0.5 for both $i=1$ and $i=2$. In the following section, we discuss how to improve the bound given by this solution by adding cuts to the DW relaxation.


\section{Consistency cuts}\label{sharingCuts}
We propose a family of valid inequalities, which we term consistency cuts, to improve the bound of a DW relaxation when two or more blocks share multiple binary variables. We define a \emph{pattern} as a set of variable value pairs. Each pair defines the value of the particular variable. The variables in a pattern all belong to $x_{0}$. Consider the example from Section \ref{example1}, an example of a pattern would be $\{(x_2,1),(x_3,0)\}$. The values of the variables in the pattern match the values in extreme points 3, 7, and 11 in the example. We say that the pattern matches these extreme points (an extreme point might include more variables than a matching pattern).

We define $\mathcal{X}_{ij}$ as the set
of binary variables from $x_{0}$ that are shared between the constraints
of $\mathcal{S}_{i}$ and $\mathcal{S}_{j}$. In the example from
Section \ref{example1} we would have $\mathcal{X}_{1,2}=\{x_{2},x_{3}\}$,
which in this case is equal to $x_{0}$ but in general does not
have to be. We define $Q_{ij}$ as the indices of the patterns between block $i$ and $j$ that use \textbf{all} variables from $\mathcal{X}_{ij}$.
Furthermore, we define $\bar{y}_{qp}^{i}$ as 1 if pattern $q\in Q_{ij}$ matches extreme point $p$ from $\mathcal{P}_{i}$ and 0 otherwise. With this we can state
\begin{proposition}\label{prop:linkingCuts}
\label{prop:LinkingCuts}The consistency cuts
\begin{alignat}{2}
\sum_{p\in\mathcal{P}_{i}}\lambda_{ip}\bar{y}_{qp}^{i}=\sum_{p\in\mathcal{P}_{j}}\lambda_{jp}\bar{y}_{qp}^{j} & \quad & \forall i,j\in\{1,\ldots,k\},i<j,q\in Q_{ij}\label{eq:sharingcuts}
\end{alignat}
can be added to \emph{($\mathsf{DW2}$)} without changing the set of feasible solutions
\end{proposition}
\noindent
In (\ref{eq:sharingcuts}), we fix $j$ to be larger than $i$ to avoid duplicate constraints. Intuitively the cuts ensure that
the weight with which pattern $q$ is used in the extreme points of
$X_{i}$ is the same as in the extreme points of $X_{j}$. As
stated, the equalities do not change the set of feasible solutions to ($\mathsf{DW2}$). They can, however, restrict the solution space of ($\mathsf{DW3}$) and therefore act as valid inequalities.
The consistency cuts are written as equalities,
while a cut (valid inequality) would be expressed as an inequality.
Obviously, constraints~\eqref{eq:sharingcuts} can each be written as two inequalities, at most one of which can be violated for a given LP solution.

\begin{proof}
Given an arbitrary $i,j\in\{1,\ldots,k\},i<j$
and $q\in Q_{ij}$ we show that the corresponding equality \eqref{eq:sharingcuts}
is valid. We do so by making an addition to the original problem ($\mathsf{IP}$)
that does not change the solution space and show that when applying 
Dantzig-Wolfe reformulation to the modified problem the desired constraint appears in the master problem. This follows an
interpretation of non-robust cuts presented
in Section 4.2 of \citep{desaulniers2011cutting}.

 We will refer to the individual variables in the $x_{0}$
vector as $x_{0}^{l}$ for $l\in\{1,\ldots,n_{0}\}$. We can divide
the variables in the pattern into two sets; $\alpha_{q}$ (resp. $\beta_{q}$)
are the variable indices of $x_{0}$ that have a value of 1 (resp.
0) in pattern $q$. We add the binary variable $y_{q}$ and the following
constraints to ($\mathsf{IP}$):
\begin{alignat}{1}
\sum_{l\in\alpha_{q}}x_{0}^{l}+\sum_{l\in\beta_{q}}\left(1-x_{0}^{l}\right)\geq(|\alpha_{q}|+|\beta_{q}|)y_{q}\label{eq:sharing_y1}\\
\sum_{l\in\alpha_{q}}x_{0}^{l}+\sum_{l\in\beta_{q}}\left(1-x_{0}^{l}\right)\leq(|\alpha_{q}|+|\beta_{q}|-1)+y_{q}\label{eq:sharing_y2}
\end{alignat}
The constraints ensure that $y_{q}$ is 1 if pattern $q$ is part
of the solution and 0 if it is not. Observe that these constraints
do not restrict the solution space of ($\mathsf{IP}$),
they simply define the value of the $y_{q}$ variable. When decomposing
the model we duplicate the constraints (\ref{eq:sharing_y1}) and
(\ref{eq:sharing_y2}) and include a copy of each in both $\mathcal{S}_{i}$
and $\mathcal{S}_{j}$. As a result, $y_{q}$ becomes a linking variable.We
can now extend (\ref{eq:convexhull}) with the $y_{q}$ variable:
\[
\widehat{X}_{iq}={\textstyle \operatorname{conv}}\left\{ \left(\begin{array}{c}
x_{0}\\
x_{i}\\
y_{q}
\end{array}\right)\in\mathbb{Z}^{n_{0}+n_{i}}\times\mathbb{B}\left|\begin{array}{c}
\left(\begin{array}{c}
x_{0}\\
x_{i}
\end{array}\right)\in\bar{X}_{i},\\
\sum_{l\in\alpha_{q}}x_{0}^{l}+\sum_{l\in\beta_{q}}\left(1-x_{0}^{l}\right)\geq(|\alpha_{q}|+|\beta_{q}|)y_{q},\\
\sum_{l\in\alpha_{q}}x_{0}^{l}+\sum_{l\in\beta_{q}}\left(1-x_{0}^{l}\right)\leq(|\alpha_{q}|+|\beta_{q}|-1)+y_{q}
\end{array}\right.\right\} 
\]
When applying Minkowski-Weyl's Theorem to the above we get
\[
\widehat{X}'_{iq}=\left\{ \left(\begin{array}{c}
x_{0}\\
x_{i}\\
y_{q}
\end{array}\right)\in\mathbb{R}^{n_{0}+n_{i}}\times[0;1]\left|\left(\begin{array}{c}
x_{0}\\
x_{i}\\
y_{q}
\end{array}\right)=\sum_{p\in\mathcal{P}_{i}}\lambda_{ip}\left(\begin{array}{c}
\bar{x}_{0p}^{i}\\
\bar{x}_{ip}\\
\bar{y}_{qp}^{i}
\end{array}\right),\sum_{p\in\mathcal{P}_{i}}\lambda_{ip}=1,\lambda_{ip}\geq0\ \forall p\in\mathcal{P}_{i}\right.\right\} 
\]
where $\left(\begin{array}{ccc}
\bar{x}_{0p}^{i} & \bar{x}_{ip} & \bar{y}_{qp}^{i}\end{array}\right)^{\intercal}$ for $p\in\mathcal{P}_{i}$ are the extreme points of $\widehat{X}_{iq}$.
Likewise ${X}_{j}$ is extended to $\widehat{X}_{jq}$ which has the
extreme points $\left(\begin{array}{ccc}
\bar{x}_{0p}^{j} & \bar{x}_{jp} & \bar{y}_{qp}^{j}\end{array}\right)^{\intercal}$ for $p\in\mathcal{P}_{j}$. Notice that because
$y_{q}$ is a linking variable just like $x_{0}$ earlier, we need
to add an extra $i$-index to identify from which $X_{i}$ it originates.
This leads to two new constraints in ($\mathsf{DW1}$):
\begin{alignat*}{1}
y_{q} & =\sum_{p\in\mathcal{P}_{i}}\lambda_{ip}\bar{y}_{qp}^{i}\\
y_{q} & =\sum_{p\in\mathcal{P}_{j}}\lambda_{jp}\bar{y}_{qp}^{j}
\end{alignat*}
Next, one of these constraints is removed by substitution of its right
hand side into the other constraint. This leaves us with
\begin{alignat*}{1}
\sum_{p\in\mathcal{P}_{i}}\lambda_{ip}\bar{y}_{qp}^{i} & =\sum_{p\in\mathcal{P}_{j}}\lambda_{jp}\bar{y}_{qp}^{j}
\end{alignat*}

\noindent which is exactly the constraint in (\ref{eq:sharingcuts}) that we wanted to show the validity of.

\end{proof}

\noindent The consistency cuts are inspired by the \emph{arc-quantity} constraints proposed by \citep{mingozzi2017exact}. The arc-quantity
constraints are defined on a single integer variable $x_{p}$ that is shared
between the constraints of two blocks $\mathcal{S}_{i}$ and $\mathcal{S}_{j}$.
An arc-quantity constraint is based on a particular value $\gamma$ that may
be assigned to $x_{p}$. The constraint ensures that the combined weight
of the extreme points of $\mathcal{S}_{i}$ that assign the value
$\gamma$ to $x_{p}$ matches the combined weight of the extreme points
of $\mathcal{S}_{j}$ that also assign the value $\gamma$ to $x_{p}$. In short, the arc-quantity constraints aim to make a single integer variable consistent across blocks, whereas the consistency cuts ensure consistency for sets of binary variables.

\subsection{Example}

\label{subsec:example}

We return to the example from Section \ref{example1}. The
example shows that equality \eqref{eq:sharingcuts} not only is valid, but that it
also is useful in eliminating fractional solutions to ($\mathsf{DW3}$). If we
look closer at the extreme points that are "used" in the solution,
we find two patterns for block 1, $\{(x_{2},0),(x_{3},0)\}$
and $\{(x_{2},1),(x_{3},1)\}$ and two different patterns
for block 2, $\{(x_{2},1),(x_{3},0)\}$ and $\{(x_{2},0),(x_{3},1)\}$.
Even though the solution satisfies (\ref{eq:MAS_split}), the same
patterns are not being used with the same weights for block 1 and
block 2. We attempt to remedy this by adding a consistency
cut on the $\{(x_{2},0),(x_{3},0)\}$ pattern. The updated
optimal LP solution is given in Table~\ref{tab:smallexample2}.

\begin{singlespace}
\begin{table}[H]
\centering %
\begin{tabular}{c|ccccc|ccccccc}
\hline 
 & \multicolumn{5}{c|}{Block 1} & \multicolumn{7}{c}{Block 2}\tabularnewline
\hline 
$\lambda$  & 0  & 1  & 0  & 0  & 0  & 0  & 0  & 0  & 1  & 0  & 0  & 0 \tabularnewline
$c$  & 0  & 3  & 2  & 2  & 4  & 0  & 0  & 0  & 3  & 0  & 3  & 3 \tabularnewline
\hline 
$\bar{x}_{1}$  & 0  & 1  & 0  & 0  & 0  & -  & -  & -  & -  & -  & -  & - \tabularnewline
$\bar{x}_{2}$  & 0  & 0  & 1  & 0  & 1  & 0  & 1  & 0  & 0  & 1  & 1  & 0 \tabularnewline
$\bar{x}_{3}$  & 0  & 0  & 0  & 1  & 1  & 0  & 0  & 1  & 0  & 1  & 0  & 1 \tabularnewline
$\bar{x}_{4}$  & -  & -  & -  & -  & -  & 0  & 0  & 0  & 1  & 0  & 1  & 1 \tabularnewline
$\bar{y}$  & 1  & 1  & 0  & 0  & 0  & 1  & 0  & 0  & 1  & 0  & 0  & 0 \tabularnewline
\hline 
\end{tabular}\caption{A small example with one consistency cut}
\label{tab:smallexample2} 
\end{table}

\end{singlespace}

\noindent The $\bar{y}$ row indicates whether or not the $\{(x_{2},0),(x_{3},0)\}$
pattern matches that extreme point, i.e. 1 means it does. The solution
above is optimal to the DW relaxation and it is integer, hence it
is also optimal to the original problem. The solution has an objective
value of 6 and shows that the bound was successfully improved by adding
the cut.
\subsection{Theoretical properties of consistency cuts}

Section~\ref{subsec:example} presents an example where the optimal solution to the DW relaxation with consistency cuts is also an optimal solution to the original problem. In the following, we prove that under certain conditions this is always the case:

\begin{theorem}
For any MIP, the optimal solution to the DW relaxation with all consistency cuts is an optimal solution to the MIP, if the following conditions are satisfied by the chosen decomposition:

\begin{enumerate}
  \item $\mathcal{S}_{0}=\emptyset$, i.e. the decomposition has no coupling constraints.
  \item All linking variables are binary.
  \item $\mathcal{X}_{ij}=\emptyset\ \forall\, i,j\in\{1,\ldots,k\}\bigl||i-j|\geq 2$
  \item $\bar{X}_{i}\,\forall\, i\in \{1,\ldots,k\}$ are bounded.
\end{enumerate}
\label{Theorem}
\end{theorem}

\noindent The third condition of Theorem \ref{Theorem} essentially means that the decomposition has to form a chain structure where $\mathcal{S}_{1}$ shares variables with $\mathcal{S}_{2}$ only, $\mathcal{S}_{2}$ with $\mathcal{S}_{1}$ and $\mathcal{S}_{3}$ only, \ldots, and $\mathcal{S}_{k}$ with $\mathcal{S}_{k-1}$ only. This would mean that the rows and columns of the constraint matrix could be reordered to form a \emph{staircase structure} \citep{fourer1984staircase}. 

Before proving Theorem \ref{Theorem} it is useful to establish some notation. We define $\mathcal{P}_{i}(q)$ as the indices of extreme points of $X_{i}$ which pattern $q$ matches. This allows us to express the consistency cuts as follows:

\begin{singlespace}
\begin{equation}
\sum_{p\in\mathcal{P}_{i}(q)}\lambda_{ip}=\sum_{p\in\mathcal{P}_{j}(q)}\lambda_{jp}\quad\forall i,j\in \{1,\ldots,k\},i<j,q\in Q_{ij}\label{eq:sharingcut_simple}
\end{equation}
\end{singlespace}

\noindent which is equivalent to (\ref{eq:sharingcuts}). Consider a general MIP of the form

\begin{singlespace}
$\,$

\noindent($\mathsf{MIP}$) \hspace{1.9in} $\mathrm{min}\{cx|Ax\leq b,x\in\dot{X}\}$

$\,$
\end{singlespace}

\noindent Where $\dot{X}$ can include both binary, integer and continuous variables. We decompose ($\mathsf{MIP}$) into $\mathcal{S}_{i}$ for $i\in\{0,\ldots,k\}$. After substitution we can compare the result to ($\mathsf{DW2}$). Since there are no constraints left in $\mathcal{S}_{0}$ (the first condition of Theorem \ref{Theorem}) constraints (\ref{eq:VS3_set1}) are not present. After relaxation and after adding the consistency cuts (\ref{eq:sharingcut_simple}) we are left with


\begin{singlespace}
$\,$

\noindent($\mathsf{DW4}$)

\begin{minipage}{\textwidth-\parindent}

\begin{equation*}
\mathrm{min}\:c_{0}\sum_{p\in\mathcal{P}_{1}}\lambda_{1p}\bar{x}_{0p}^{1}+\sum_{i=1}^{k}c_{i}\sum_{p\in\mathcal{P}_{i}}\lambda_{ip}\bar{x}_{ip}
\end{equation*}

subject to

\begin{alignat}{2}
\sum_{p\in\mathcal{P}_{i}(q)}\lambda_{ip}-\sum_{p\in\mathcal{P}_{j}(q)}\lambda_{jp} & =0 & \quad & \forall i,j\in\{1,\ldots,k\},i<j,q\in Q_{ij}\label{eq:sharingcut_simple2}\\
\sum_{p\in\mathcal{P}_{i}}\lambda_{ip}\bar{x}_{0p}^{i}-\sum_{p\in\mathcal{P}_{i+1}}\lambda_{i+1,p}\bar{x}_{0p}^{i+1} & =0 & \quad & \forall i\in\{1,\ldots,k-1\}\label{eq:VS3_split2}\\
\sum_{p\in\mathcal{P}_{i}}\lambda_{ip} & =1 & \quad & \forall i\in\{1,\ldots,k\}\label{eq:VS3_conv1_2}\\
\lambda_{ip} & \geq0 & \quad & \forall p\in\mathcal{P}_{i},i\in\{1,\ldots,k\}\label{eq:VS3_conv2_2}
\end{alignat}

\end{minipage}

$\,$
\end{singlespace}

\noindent Notice that (\ref{eq:VS3_split2}) is slightly different from (\ref{eq:VS3_split}), but they define the same solution space. Rather than comparing all linking variable copies to the copies from the first block, we compare the copies from block 1 to those of block 2, 2 to 3, etc. To prove Theorem \ref{Theorem} we establish that ($\mathsf{DW4}$) is equivalent to a smaller LP and that the constraint matrix of the smaller LP is totally unimodular. This means that any basic solution to the smaller LP is naturally integer. 
Of the following lemmas, Lemma \ref{WeJustNeedOneConv} states that all but one of the constraints (\ref{eq:VS3_conv1_2}) are redundant and Lemma \ref{WeDoNotNeedVarSplitSingleVar} states that all constraints ({\ref{eq:VS3_split}}) are redundant while Lemma \ref{SumLemma} is  used to prove Lemma \ref{WeJustNeedOneConv}.

\begin{lemma}
Let $i,j\in\{1,\ldots,k\}$. Assume $i\neq j$ and that $\mathcal{X}_{ij}\neq\emptyset$ then\label{SumLemma}
\end{lemma}

\begin{singlespace}
\[
\sum_{p\in\mathcal{P}_{i}}\lambda_{ip}=\sum_{q\in Q_{ij}}\sum_{p\in\mathcal{P}_{i}(q)}\lambda_{ip}
\]
\end{singlespace}

\begin{proof}
Every index in $\mathcal{P}_{i}$ corresponds to an extreme point from $X_{i}$. Such an extreme point will match exactly one pattern from $Q_{ij}$.
This means that every $p$ from $\mathcal{P}_{i}$ is going to occur exactly once in the summation on the right-hand side and thereby we are summing over exactly the same variables on the left and right-hand side.
\end{proof}

\begin{lemma}
Constraints (\ref{eq:VS3_conv1_2}) can be removed from \emph{($\mathsf{DW4}$)} for every value of $i$ other than 1 without changing the solution space.\label{WeJustNeedOneConv}
\end{lemma}

\begin{proof}
Assume that 
\[
\sum_{p\in\mathcal{P}_{i}}\lambda_{ip}=1
\]
holds for an $i\in\{1,\ldots,k-1\}$. Then the following series of equations show that it also holds for $i+1$.

\begin{singlespace}
\begin{align*}
1 & =\sum_{p\in\mathcal{P}_{i}}\lambda_{ip}\overset{(a)}{=}\sum_{q\in Q_{i.i+1}}\sum_{p\in\mathcal{P}_{i}(q)}\lambda_{ip}\overset{(b)}{=}\sum_{q\in Q_{i.i+1}}\sum_{p\in\mathcal{P}_{i+1}(q)}\lambda_{i+1p}\overset{(c)}{=}\sum_{p\in\mathcal{P}_{i+1}}\lambda_{i+1p}
\end{align*}
\end{singlespace}

\noindent (a) follows directly from Lemma \ref{SumLemma}. Since all consistency cuts for $q\in Q_{i,i+1}$ are assumed to be present, (b) is simply a substitution based on the consistency cuts (\ref{eq:sharingcut_simple2}) for all $q\in Q_{i,i+1}$. Finally, (c) follows from Lemma \ref{SumLemma} after realizing that $Q_{i,i+1}=Q_{i+1,i}$, which follows directly from the definition of these. Since constraints (\ref{eq:VS3_conv1_2}) were not removed for $i=1$, we now have the base case and the induction step and can prove by induction that the constraints hold for any value of $i\in\{1,\ldots,k\}$.
\end{proof}

\begin{lemma}
The variable linking constraints (\ref{eq:VS3_split}) can be removed from \emph{($\mathsf{DW4}$)} without changing the solution space.\label{WeDoNotNeedVarSplitSingleVar}
\end{lemma}
\begin{proof} 
Recall that $\bar{x}_{0p}^{i}$ is a vector, we denote
the $l^{th}$ element of this vector $\bar{x}_{0p}^{il}$.
This allows us to write (\ref{eq:VS3_split}) in the following way:
\begin{singlespace}
\begin{alignat*}{2}
\sum_{p\in\mathcal{P}_{i}}\lambda_{ip}\bar{x}_{0p}^{il}-\sum_{p\in\mathcal{P}_{i+1}}\lambda_{i+1.p}\bar{x}_{0p}^{i+1.l} & =0 & \quad & \forall i\in\{1,\ldots,k-1\},l\in\mathcal{X}_{i,i+1}
\end{alignat*}
\end{singlespace}
We abuse notation slightly in the equation above. $\mathcal{X}_{i,j}$ is actually a set of variables, but here we use it as a set of variable indices.
We look at an arbitrary $i\in\{1,\ldots,k-1\}$ and $l\in\mathcal{X}_{i,i+1}$.
Let $Q_{i,i+1}(l)$ be the patterns from $Q_{i,i+1}$ that assigns
value 1 to the $l^{th}$ variable from $x_{o}$. Then
\begin{singlespace}
\begin{align*}
\sum_{p\in\mathcal{P}_{i}}\lambda_{ip}\bar{x}_{0p}^{il}\overset{(a)}{=}\sum_{p\in\mathcal{P}_{i}|\bar{x}_{0p}^{il}=1}\lambda_{ip}\overset{(b)}{=}\sum_{\begin{array}{c}
q\in Q_{i,i+1}(l)\end{array}}\sum_{\mathcal{P}_{i}(q)}\lambda_{ip}\overset{(c)}{=}\sum_{\begin{array}{c}
q\in Q_{i,i+1}(l)\end{array}}\sum_{\mathcal{P}_{i+1}(q)}\lambda_{ip}\\
\overset{(d)}{=}\sum_{p\in\mathcal{P}_{i+1}|\bar{x}_{0p}^{i+1,l}=1}\lambda_{i+1,p}\overset{(e)}{=}\sum_{p\in\mathcal{P}_{i+1}}\lambda_{i+1,p}\bar{x}_{0p}^{i+1,l}
\end{align*}
\end{singlespace}
\noindent (a) holds because the $x$-variables are binary (the second
condition of Theorem \ref{Theorem}). (b) follows from the definition
of $Q_{i,i+1}(l)$ and $\mathcal{P}_{i}(q)$. (c) follows directly
from the definition of the consistency cuts. (d) is
similar to (b) and (e) is similar to (a). We could do the same for
any other  $i\in\{1,\ldots,k-1\}$ and $l\in\mathcal{X}_{i,i+1}$, this
proves the lemma. \end{proof}

\noindent Applying Lemma \ref{WeJustNeedOneConv} and Lemma \ref{WeDoNotNeedVarSplitSingleVar} to ($\mathsf{DW4}$) reduces the formulation to 

\begin{singlespace}
$\,$

\noindent($\mathsf{DW5}$)

\begin{minipage}{\textwidth-\parindent}

\begin{equation*}
\mathrm{min}\:c_{0}\sum_{p\in\mathcal{P}_{1}}\lambda_{1p}\bar{x}_{0p}^{1}+\sum_{i=1}^{k}c_{i}\sum_{p\in\mathcal{P}_{i}}\lambda_{ip}\bar{x}_{ip}
\end{equation*}

\begin{alignat}{2}
\sum_{p\in\mathcal{P}_{i}}\lambda_{ip}\bar{y}_{qp}^{i}-\sum_{p\in\mathcal{P}_{j}}\lambda_{jp}\bar{y}_{qp}^{j} & =0 & \quad & \forall i,j\in\{1,\ldots,k\},i<j,q\in Q_{ij}\label{eq:sharingcuts2}\\
\sum_{p\in\mathcal{P}_{1}}\lambda_{1p} & =1 \label{simpleConv}\\
\lambda_{ip} & \geq0 & \quad & \forall p\in\mathcal{P}_{i},i\in\{1,\ldots,k\}\label{eq:VS3_conv2_3}
\end{alignat}

\end{minipage}

$\,$
\end{singlespace}

\noindent Notice that we have replaced (\ref{eq:sharingcut_simple2}) with its original form (\ref{eq:sharingcuts}), ((\ref{eq:sharingcuts2}) in the model above). These are freely interchangeable. We could also write (\ref{eq:sharingcuts2})-(\ref{eq:VS3_conv2_3}) as follows

\begin{singlespace}
\begin{alignat}{1}
\bar{A}\lambda & =\bar{b}\label{eq:DW6_cons}\\
\lambda & \geq0\label{eq:DW6_dom}
\end{alignat}
\end{singlespace}

\noindent with the appropriate definitions of $\bar{A}$ and $\bar{b}$. We continue by proving that $\bar{A}$ is \emph{totally unimodular} (TU), and therefore show that the reformulation has integer extreme points (see Proposition III.2.3). For this the following corollary is necessary.

\begin{corollary}\textbf{III.2.8}
\citep{wolsey1988integer} Let A be a (0, 1, -1) matrix with no more than two nonzero elements in each column. Then A is TU if and only if the rows of A can be partitioned into two subsets $\mathcal{Q}_1$ and $\mathcal{Q}_2$ such that if a column contains two nonzero elements, the following statements are true:

\begin{enumerate}
  \item If both nonzero elements have the same sign, then one is in a row contained in $\mathcal{Q}_1$ and the other is in a row contained in $\mathcal{Q}_2$.
  \item If the two nonzero elements have the opposite sign, then both are in rows contained in the same subset.
\end{enumerate}
\end{corollary}


\noindent We define, for each $i,j\in\{1,\ldots,k\}$ where $i\neq j$, $\hat{y}_{ij}$ as the matrix of the $\bar{y}_{qp}^{i}$ parameters, where every row corresponds to a $q\in Q_{ij}$ (recall $Q_{ij}=Q_{ji}$) and every column corresponds to a $p\in\mathcal{P}_i$. Because of the structure specified by the third condition of Theorem \ref{Theorem}, we can represent $\bar{A}$ in the following way

\begin{singlespace}
\begin{equation}
\begin{aligned}\bar{A}= & \left[\begin{array}{cccccc}
\hat{y}_{1,2} & -\hat{y}_{2,1}\\
& \hat{y}_{2,3} & -\hat{y}_{3,2}\\
&  & \hat{y}_{3,4}\\
&  &  & \ddots\\
&  &  &  & -\hat{y}_{k-1,k-2}\\
&  &  &  & \hat{y}_{k-1,k} & -\hat{y}_{k,k-1}\\
\textbf{1} &  &  &  &
\end{array}\right]\end{aligned}\nonumber
\end{equation}
\end{singlespace}

\noindent This has a staircase-like structure since variables are only shared between neighboring blocks. \textbf{1} is a row vector consisting of $|\mathcal{P}_{1}|$ 1s, and since the $y$ coefficients are binary, all the elements in $\bar{A}$ have values of -1, 0 or 1. It is important to realize that each of the $\hat{y}_{ij}$ matrices have a single 1 in each column; only a single pattern from $Q_{ij}$ (the maximal patterns) matches each extreme point in $P_i$, and since each column in  $\hat{y}_{ij}$ corresponds to an extreme point in $P_i$, each column has a 1 for exactly the matching pattern. There is one exception to this, if some of the consecutive $\mathcal{S}_{i}$, $\mathcal{S}_{i+1}$ do not share variables the corresponding $\hat{y}_{i,i+1}$ and $\hat{y}_{i+1,i}$ will only consist of zeros. This means that every column in $\bar A$ has at most two nonzero elements.

We partition $\bar{A}$ according to Corollary III.2.8 into $\mathcal{Q}_1$ and $\mathcal{Q}_2$ such that $\mathcal{Q}_2$ consists only of the last row in $\bar{A}$ and $\mathcal{Q}_1$ of all remaining rows. This means that for all columns where $i\in\{2,\ldots,k-1\}$ there will be at most one 1 (from $\hat{y}_{i,i-1}$) and one -1 (from -$\hat{y}_{i,i+1}$), both of which are in $\mathcal{Q}_1$. This falls under the second case of the corollary. For $i=1$ each column will have at most two 1s (one from $\hat{y}_{1,2}$ and one from \textbf{1}), one in each of the partitions. This is in line with the first case of the corollary. Finally, the columns for $i=k$ have at most a single 1. For these and any of the other columns that might have fewer than 2 nonzero elements, the corollary is trivially satisfied. Hence, we conclude that $\bar{A}$ is TU.

The final step to prove Theorem \ref{Theorem} relies on the following proposition:

\begin{propos}\textbf{III.2.3}
\citep{wolsey1988integer} If $A$ is $TU$, if $b$, $b'$, $d$, and $d'$ are integral, and if $P(b,b',d,d')=\{x\in \mathbb{R}^n: b'\leq Ax\leq b, d'\leq x\leq d\}$ is not empty, then $P(b,b',d,d')$ is an integral polyhedron.
\end{propos}

\noindent (\ref{eq:DW6_cons}) and (\ref{eq:DW6_dom}) fit directly into Proposition III.2.3 and since $\bar{A}$ is TU and $\bar{b}$ is integral (consists of 1s and 0s) it follows from the proposition that an optimal basic solution is integer.

The solution space of ($\mathsf{DW5}$) is exactly the same as that of ($\mathsf{DW4}$) due to lemmas \ref{WeJustNeedOneConv} and \ref{WeDoNotNeedVarSplitSingleVar}. It follows that an optimal basic solution to ($\mathsf{DW4}$) is integer and this concludes the proof of Theorem \ref{Theorem}.

The conditions of Theorem \ref{Theorem} might seem limiting, but observe that for any \emph{binary integer problem} (BIP) we can construct numerous decompositions that satisfy the conditions of the theorem; there are many ways to split the constraints into two sets $\mathcal{S}_{1}$ and $\mathcal{S}_{2}$ leaving $\mathcal{S}_{0}$ empty.

\begin{corollar}
Theorem \ref{Theorem} can be generalized by relaxing the second requirement to include bounded integer variables, this is done by substituting each integer linking variable with a set of binary variables.\label{BoundedIntegerVars}
\end{corollar}

\noindent Consider an integer variable $x$ with upper bound $UB$ and lower bound $LB$. We define the parameter $m=\left \lfloor{log_2(UB-LB)}\right \rfloor+1$ and binary variables $z_i$ for $i\in\{0,\ldots,m\}$, this allows us to define $x$ as:

\begin{singlespace}
\begin{equation}
x=LB+\sum_{i=0}^{m}2^{i}z_i\label{eq:intbinsub}
\end{equation}
\end{singlespace}

\noindent We can remove the lower bound constraint for $x$ ($x\geq LB$), but we need to keep the upper bound constraint ($x\leq UB$). By substituting $x$ with the right-hand side of (\ref{eq:intbinsub}) we can replace the integer variable with a set of binary linking variables. This satisfies the second condition of Theorem \ref{Theorem}.

\begin{corollar}
Theorem \ref{Theorem} can be further generalized by rephrasing its third condition as such:

\begin{enumerate}
  \setcounter{enumi}{2}
  \item If a variable is shared between two blocks $\mathcal{S}_{i}$ and $\mathcal{S}_{j}$, where $i,j\in\{1,\ldots,k\}$ and $j-i\geq 2$, the variable must also be shared with all intermediate blocks $\mathcal{S}_{t}$ for $t\in\{i+1,\ldots,j-1\}$.
\end{enumerate}

\label{IntermediateSets}
\end{corollar}

\noindent Corollary \ref{IntermediateSets} will follow easily from  Lemma \ref{WeOnlyNeedIntermediateCuts} below. To prove 
Lemma \ref{WeOnlyNeedIntermediateCuts} we need some further notation and Lemma \ref{OnlyMaximalIsNeeded}.
Earlier $Q_{ij}$ was defined as patterns involving all variables from $\mathcal{X}_{ij}$. It is also possible to define patterns on
a subset of variables from  $\mathcal{X}_{ij}$. All such patterns are denoted $\tilde{Q}_{ij}$ (so $Q_{ij} \subseteq \tilde{Q}_{ij}$). We write $q' \subset q$ for $q' \in \tilde{Q}_{ij}$
and $q \in Q_{ij}$ if the setting of variable values in $q'$ matches the setting in $q$ for the  variables they have in common. For $q' \in \tilde{Q}_{ij}$ we extend the notation $\mathcal{P}_i(q')$ to be the indices of extreme points of $X_{i}$ which pattern $q'$ matches.
We can now state

\begin{lemma}
Given $i,j\in\{1,\ldots,k\},i<j$ . For any $q'\in\tilde{Q}_{ij} \setminus Q_{ij}$ the equality

\begin{singlespace}
\[
\sum_{p\in\mathcal{P}_{i}(q')}\lambda_{ip}=\sum_{p\in\mathcal{P}_{j}(q')}\lambda_{jp}
\]
\end{singlespace}

\noindent is implied by the consistency cuts 

\begin{singlespace}
\begin{equation*}
\sum_{p\in\mathcal{P}_{i}(q)}\lambda_{ip}=\sum_{p\in\mathcal{P}_{j}(q)}\lambda_{jp}\quad\forall q\in Q_{ij}
\end{equation*}
\end{singlespace}
\label{OnlyMaximalIsNeeded}
\end{lemma}

\begin{proof}
The lemma follows from the equations below.

\begin{singlespace}
\[
\sum_{p\in\mathcal{P}_{i}(q')}\lambda_{ip}\overset{(a)}{=}\sum_{\begin{array}{c}
q\in Q_{ij}|q'\subseteq q
\end{array}}\sum_{\mathcal{P}_{i}(q)}\lambda_{ip}\overset{(b)}{=}\sum_{\begin{array}{c}
q\in Q_{ij}|q'\subseteq q
\end{array}}\sum_{\mathcal{P}_{j}(q)}\lambda_{jp}\overset{(c)}{=}\sum_{p\in\mathcal{P}_{j}(q')}\lambda_{jp}
\]
\end{singlespace}

\noindent (a) holds because $\mathcal{P}_{i}(q)$ for $q\in Q_{ij}$ where $q'\subseteq q$ is a partitioning of $\mathcal{P}_{i}(q')$. (b) follows directly from the definition of the consistency cuts and (c) follows from the same arguments as (a).
\end{proof}

\begin{lemma} Assume 3. from Corollary \ref{IntermediateSets}; given arbitrary $i,j\in\{1,\ldots,k\}\bigl|j-i\geq 2$ and $q\in Q_{ij}$ the corresponding consistency cut
\begin{singlespace}
\begin{equation}
\sum_{p\in\mathcal{P}_{i}(q)}\lambda_{ip}=\sum_{p\in\mathcal{P}_{j}(q)}\lambda_{jp}\label{singlesharingcut}
\end{equation}
\end{singlespace}
\noindent can be removed from \emph{($\mathsf{DW4}$)} without changing the solution space.\label{WeOnlyNeedIntermediateCuts}
\end{lemma}

\begin{proof}
The lemma follows from the equations below

\begin{singlespace}
\[
\sum_{p\in\mathcal{P}_{i}(q)}\lambda_{ip}\overset{(a)}{=}\sum_{\mathcal{P}_{i+1}(q)}\lambda_{ip}\overset{(b)}{=}\ldots\overset{(c)}{=}\sum_{\mathcal{P}_{j-1}(q)}\lambda_{jp}\overset{(d)}{=}\sum_{p\in\mathcal{P}_{j}(q)}\lambda_{jp}
\]
\end{singlespace}
\noindent Because of the assumption that all shared variables between $i$ and $j$ are also shared with $i+1$, it follows that $q\in\tilde{Q}_{i,i+1}$ (but $q \in Q_{i,i+1}$ may not be true). (a) now follows from Lemma \ref{OnlyMaximalIsNeeded} and the assumption that all consistency cuts based on $Q_{i,i+1}$ are present. (b) and (c) denote the connections through all the intermediate blocks between $i$ and $j$. These also follow from Lemma~\ref{OnlyMaximalIsNeeded} and the consistency cuts, and so does (d) with a similar argument as (a). Seeing that (\ref{singlesharingcut}) is trivially satisfied by the consistency cuts of the intermediate blocks, (\ref{singlesharingcut}) can be removed.
\end{proof}

\noindent By applying Lemma \ref{WeOnlyNeedIntermediateCuts} to all consistency cuts (\ref{eq:sharingcut_simple2}) from ($\mathsf{DW4}$) where $j-i\geq 2$ we can remove them, which leaves us with just the cuts where $|j-i|<2$. This brings the problem to the same form as specified in the third condition of Theorem \ref{Theorem} and hence proves Corollary \ref{IntermediateSets}. We refer to this structure as \emph{extended chain structure}.

Given a problem and its decomposition it can easily be checked if the extended chain structure is present as long as the blocks are ordered correctly. One can simply run through all variables and check that they are only shared between consecutive blocks. If only the constraint matrix is given, detecting the structure seems a lot more difficult. One could start by detecting a staircase structure in the constraint matrix \citep{jayakumar1994clustering}. That would correspond to the basic chain structure (not extended); however, there is no guarantee that the detected structure adheres to the other criteria of Theorem~\ref{Theorem}. This is an interesting topic for further research.

In general, the consistency cuts can be used in a branch-and-cut-and-price algorithm to provide optimal integer solutions. However, if the decomposition satisfies the conditions of Theorem \ref{Theorem} or the conditions of the generalizations from Corollary \ref{BoundedIntegerVars} and Corollary \ref{IntermediateSets}, the root node will be integer, and no branching will be needed.
 
Notice that if two blocks share a variable that some or all of the intermediate blocks do not, we can simply add that variable manually to the intermediate blocks where it is missing. This is done by introducing new variable linking constraints in the master problem corresponding to the blocks where the variable is missing. The dual values of these variable linking constraints can then be used in the appropriate subproblem objectives as the coefficients of the missing variables. Essentially, we simply act as if the variables were shared by all the intermediate blocks to begin with.

This extends Theorem \ref{Theorem} even further to now include $\textbf{any}$ decomposition as long as all linking variables are either binary or bounded integers (Corollary \ref{BoundedIntegerVars}) and $\mathcal{S}_0=\emptyset$ (we can always force $\mathcal{S}_0=\emptyset$ by changing the decomposition slightly to create a new block that contains the constraints of $\mathcal{S}_0$).
This might have potential for some problems; however, in practice, we might need to binarize a lot of integer variables (Corollary \ref{BoundedIntegerVars}) or add a lot of variables manually (Corollary \ref{IntermediateSets}). We currently lack experimental results for this, but it is something that should definitely be explored in the future as it could be an alternative to branching for some problems.

\section{Temporal knapsack problem}\label{TKP}

The goal of the TKP is to choose a profit-maximizing subset out of $n$ items to put in a "knapsack". Each item $i$ has a weight $w_i$, a profit $p_i$, and a time interval of activity $[s_i,t_i[$. Notice that an item is considered active at the start time of its activity interval but not at the end time. The chosen items have to satisfy that at any point in time the weight of the active items is less than or equal to the capacity $C$ of the knapsack. 

We define binary decision variables $x_i$, which are 1 if and only if item $i$ is chosen, and parameters $S_j:=\{i\in\{1,\ldots,n\}|s_i\leq s_j,t_i>s_j\}$ as the items that are active at the start time of item $j$. This allows us to write the model:

\begin{singlespace}
$\,$

\noindent($\mathsf{TKP}$)

\begin{minipage}{\textwidth-\parindent}

\begin{equation}
\mathrm{max}\:\sum_{i=1}^{n}p_{i}x_{i}\label{eq:TKP_obj}
\end{equation}

subject to

\begin{alignat}{2}
\sum_{i\in S_{j}}w_{i}x_{i} & \leq C & \quad & \forall j\in\{1,\ldots,n\}\label{eq:TKP_con}\\
x_{i} & \in\{0,1\} & \quad & \forall i\in\{1,\ldots,n\}\label{eq:TKP_dom}
\end{alignat}

\end{minipage}

$\,$
\end{singlespace}

\noindent The objective function, (\ref{eq:TKP_obj}) maximizes the total profit. Constraints (\ref{eq:TKP_con}) ensure that the capacity is respected at all times. Notice that we only need to check at the start time of each of the items as these are the only times when items become active and can potentially violate the capacity of the knapsack. It is possible to reduce the model by eliminating some of the capacity constraints. We say that a set $S_j$ is dominated if there exists another set $S_m$ for which $S_j\subseteq S_m$; dominated constraints will be trivially satisfied and can hence be removed. After removing all dominated constraints we are left with the set $N:=\{j\in\{1,\ldots,n\}|S_j\setminus S_m\neq\emptyset\,\;\forall m\in\{1,\ldots,n\},m\neq j\}$ and can modify (\ref{eq:TKP_con}) accordingly:

\begin{singlespace}
\begin{alignat}{2}
\sum_{i\in S_{j}}w_{i}x_{i} & \leq C & \quad & \forall j\in N\label{eq:TKP_con2}
\end{alignat}
\end{singlespace}

\subsection{Decomposition}

We apply DW reformulation in the same way as \citep{caprara2013uncommon}. In the terms used in this paper, this decomposition splits the constraints into the sets $\mathcal{S}_{i}$ for $i\in\{0,\ldots,k\}$ where $\mathcal{S}_{0}=\emptyset$. Furthermore, the constraints are ordered in the decomposition according to the time interval starting times. In other words, consider two constraints $j$ and $l$ from (\ref{eq:TKP_con2}) where $s_l>s_j$, if $j$ is in $\mathcal{S}_{i}$ and $l$ is in $\mathcal{S}_{m}$ then $m\geq i$. If starting times are identical, the ties are broken arbitrarily and this does not affect the structure. Recall example (\ref{eq:Ex_obj})-(\ref{eq:Ex_dom}), which is in fact a simple TKP. The decomposition used in the example is the same as the one from \citep{caprara2013uncommon}.

This means that the decomposition satisfies the four conditions of Theorem \ref{Theorem}. The first two and the last are easily verified ($\mathcal{S}_{0}=\emptyset$ and the model only has binary variables); however, the third one might be slightly harder to see. Each constraint represents a point in time, and the constraints are chronologically ordered. Because each variable is active only in one continuous time interval, the variables follow exactly the structure specified in Corollary \ref{IntermediateSets}. This means, by adding the consistency cuts to the DW relaxation of the TKP, we guarantee that the optimal solution of the DW relaxation is integer and no branching will be necessary.

Figure \ref{ConstraintMat} shows an example of a constraint matrix of a small TKP. The colored (green and black) cells are non-zero variable coefficients. The blue lines indicate where a decomposition with 16 constraints in each block would split the problem. Some of the variables overlap multiple blocks. These are the linking variables and are indicated by the color green rather than black. Figure~\ref{ConstraintMat} shows the variables are active in continuous time intervals. If the constraints were ordered differently, then it might not be immediately obvious that Corollary \ref{IntermediateSets} applies.

\begin{figure}
\centering
\includegraphics[width=3.25in]{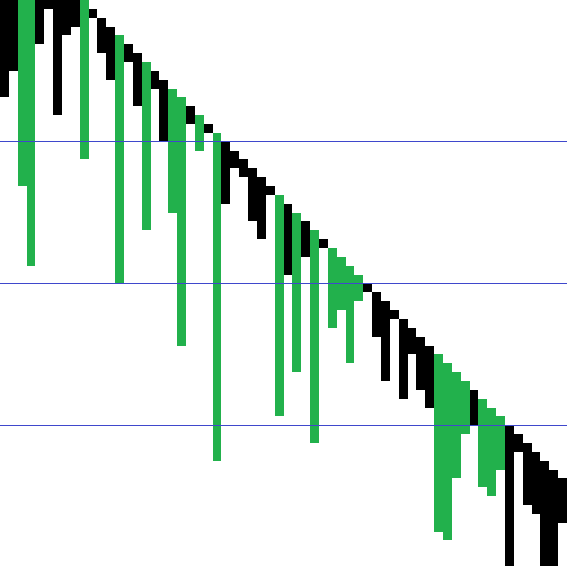}
\caption{Visualization of the constraint matrix of a small TKP. Colored (green and black) cells are non-zero variable coefficients. The blue lines indicate where the decomposition splits the constraints and the color green is used to indicate the linking variables.}
\label{ConstraintMat}
\end{figure}

\section{Implementation}\label{Implementation}

The DW relaxation with consistency cuts is used in a branch-and-cut-and-price framework to obtain optimal solutions to MIPs. This entails column generation, dynamically separating the consistency cuts, branching, and the choice of initial solution, all of which is explained below. 

\subsection{Column Generation}

DW reformulation typically results in a formulation with a large numbers of variables. Column generation is a technique for solving large scale LPs. It solves a master problem (here the DW relaxation ($\mathsf{DW3}$), with the consistency cuts (\ref{eq:sharingcuts})) by iteratively generating the columns/variables of that problem. The algorithm alternates between solving the \emph{restricted master problem} and one or more subproblems, often called pricing problems. We denote $\bar{\mathcal{P}}_{i}\subseteq\mathcal{P}_{i}$ such that $\{(\begin{array}{c c c}
\bar{x}_{0p}^{i} & \bar{x}_{ip}
\end{array})^\intercal\}_{p\in\bar{\mathcal{P}}_{i}}$ is a subset of the extreme points of $X_{i}$. The restricted master problem is given by:

\begin{singlespace}
$\,$

\noindent($\mathsf{RMP}$)

\begin{minipage}{\textwidth-\parindent}
\begin{equation}
\mathrm{min}\:c_{0}\sum_{i=1}^{k}\sum_{p\in\bar{\mathcal{P}}_{i}}\lambda_{ip}\frac{\bar{x}_{0p}^{i}}{k}+\sum_{i=1}^{k}c_{i}\sum_{p\in\bar{\mathcal{P}}_{i}}\lambda_{ip}\bar{x}_{ip}\label{eq:MAS_obj}
\end{equation}

subject to

\begin{alignat}{3}
\sum_{p\in\bar{\mathcal{P}}_{i}(q)}\lambda_{ip}-\sum_{p\in\bar{\mathcal{P}}_{j}(q)}\lambda_{jp} & =0 & \quad & \forall i,j\in\{1,\ldots,k\},i<j,q\in Q_{ij} & \quad & (\sigma_q)\label{eq:sharingcut_simple3}\\
H\sum_{p\in\bar{\mathcal{P}}_{1}}\lambda_{1p}\bar{x}_{0p}^{1}+\sum_{i=1}^{k}E_{i}\sum_{p\in\bar{\mathcal{P}}_{i}}\lambda_{ip}\bar{x}_{ip} & \leq b_{0} &  &  & \quad & (\kappa)\label{eq:MAS_set1}\\
\sum_{p\in\bar{\mathcal{P}}_{1}}\lambda_{1p}\bar{x}_{0p}^{1}-\sum_{p\in\bar{\mathcal{P}}_{i}}\lambda_{ip}\bar{x}_{0p}^{i} & =0 & \quad & \forall i\in\{2,\ldots,k\} & \quad & (\omega_{i})\label{eq:MAS_split}\\
\sum_{p\in\bar{\mathcal{P}}_{i}}\lambda_{ip} & =1 & \quad & \forall i\in\{1,\ldots,k\} & \quad & (\pi_{i})\label{eq:MAS_conv1}\\
\lambda_{ip} & \geq0 & \quad & \forall p\in\bar{\mathcal{P}_{i}},i\in\{1,\ldots,k\}\label{eq:MAS_conv2}
\end{alignat}

\end{minipage}

$\,$
\end{singlespace}

\noindent The restricted master problem is a reduced problem because we do not include all of the variables. Notice that the objective is modified slightly, the first term in (\ref{eq:MAS_obj}) is different from that in (\ref{eq:VS3_obj}). Rather than substituting the $x_0$ variables in the objective with the copies from the first subproblem, we use a weighted average across all subproblems. Preliminary tests showed that this change improves the convergence of column generation. If some linking variables only appear in the constraints of some blocks, then the $\frac{1}{k}$ coefficient would be modified accordingly.

We define the dual variables for each of the constraints in $RMP$. Let $\sigma_q$ for $q\in Q_{ij}$ be the dual variables associated with (\ref{eq:sharingcut_simple3}). $\kappa$ denotes the vector of dual variables associated with  (\ref{eq:MAS_set1}); it has an element for each row in $E$. For each $i\in\{2,\ldots,k\}$, $\omega_{i}$ is the vector of dual variables associated with (\ref{eq:MAS_split}) in subproblem $i$, it has an element for each element in $x_0$, i.e. $n_0$ elements. Finally, For each $i\in\{1,\ldots,k\}$, $\pi_{i}$ is the dual variable associated with  (\ref{eq:MAS_conv1}) in subproblem $i$. New columns for the restricted master problem are generated by solving the subproblems. The $i^{th}$ subproblem for $i>1$ is:

\begin{singlespace}
$\,$

\noindent

\begin{minipage}{\textwidth-\parindent}
\begin{equation}
\bar{c}_{i}=\mathrm{min}\:
\frac{c_{0}}{k}{x}_{0} + c_{i}{x}_{i}
-\kappa E_{i}x_{i}+\omega_{i}x_{0}-\pi_{i}
+\sum_{j\in\{1,\ldots,i-1\}}\sum_{q\in Q_{ij}}\sigma_q y_q
-\sum_{j\in\{i+1,\ldots,k\}}\sum_{q\in Q_{ij}}\sigma_q y_q\label{eq:SUB1_obj}
\end{equation}

subject to
\begin{alignat}{3}
F_{i}x_{0}+D_{i}x_{i} & \leq b_{i}\label{eq:SUB1_con}\\
\sum_{l\in\alpha_{q}}x_{0}^{l}+\sum_{l\in\beta_{q}}\left(1-x_{0}^{l}\right) & \geq(|\alpha_{q}|+|\beta_{q}|)y_{q} & \quad & \forall i,j\in\{1,\ldots,k\},i<j,q\in Q_{ij}\label{eq:sharing_y1_2}\\
\sum_{l\in\alpha_{q}}x_{0}^{l}+\sum_{l\in\beta_{q}}\left(1-x_{0}^{l}\right) & \leq(|\alpha_{q}|+|\beta_{q}|-1)+y_{q} & \quad & \forall i,j\in\{1,\ldots,k\},i<j,q\in Q_{ij}\label{eq:sharing_y2_2}\\
x_{0} & \in\mathbb{Z}^{n_{0}}\label{eq:SUB2_dom1}\\
x_{i} & \in\mathbb{Z}^{n_{i}}\label{eq:SUB2_dom2}\\
y_{q} & \in\{0,1\} & \quad & \forall i,j\in\{1,\ldots,k\},i<j,q\in Q_{ij}
\end{alignat}

\end{minipage}

$\,$
\end{singlespace}

\noindent Constraints (\ref{eq:SUB1_con}) are the constraints in block $i$ from (\ref{eq:DW1_blocks}). Constraints (\ref{eq:sharing_y1_2}) and (\ref{eq:sharing_y2_2}) are simply (\ref{eq:sharing_y1}) and (\ref{eq:sharing_y2}) applied to the subproblem; these define the $y_q$ variables. We need these variables to keep track of which patterns match each extreme point. The objective (\ref{eq:SUB1_obj}) defines the reduced cost $\bar{c}_{l}$ of the subproblem. If a subproblem has a solution with a negative reduced cost, the corresponding column is added to the DW relaxation. The first two terms of (\ref{eq:SUB1_obj}) are simply the original objective of $\mathsf{IP}$. The $x_0$ part is split evenly over the subproblems, and the $x_i$ part appears only in subproblem $i$. The rest of the objective deals with the dual values of the constraints in $(\mathsf{RMP})$. In the special case of ${i}=1$, we add the term $-\kappa Hx_0$ to the objective (\ref{eq:SUB1_obj}).


In each iteration, the restricted master problem and a subset of the subproblems are solved. Each subproblem is solved with CPLEX and a (deterministic) time limit $\tau$ measured in \emph{ticks} (see \citep{CPLEX-ticks}). If a subproblem is solved to optimality within the limit and the optimal solution has negative reduced cost, the corresponding column is added to the restricted master problem. When a subproblem does not produce a column with negative reduced cost (denoted a "fail") it is skipped in the following iterations. The number of iterations the subproblem is skipped is dependent on the number of consecutive fails of the subproblem. This is done to speed up the column generation process. When none of the solved subproblems provides a solution with negative reduced cost, column generation would normally terminate with an optimal solution to the master problem. However, because of time limits and the skipping of subproblems, this is not necessarily the case. To ensure optimality, the time limit is increased by a factor $\eta$, and any subproblems that were not solved within the old limit are rerun together with any skipped subproblems. The time limit is increased until at least one subproblem provides a column with negative reduced cost, in which case column generation has to continue; or until all subproblems are solved to optimality without a single one finding such a column, in which case the solution is optimal.

The column generation algorithm is stabilized using \textit{box penalty stabilization} (BPS) of the dual values \citep{du1999stabilized}. This technique relies on an initial estimate of the dual variables. A bounding box is placed around each estimate and deviations from this box are linearly penalized. The estimated dual variable values are then updated at each iteration. BPS has been shown to significantly improve the convergence of column generation. In our implementation a box of width 0 is used, i.e. there is essentially no box around the estimate in which we do not penalize deviations, rather all deviations are penalized linearly no matter how small they are \citep{pigatti2005stabilized}. The penalty factor is initially set to $\rho$. When the stabilized master problem is optimal, the penalty is reduced with a factor of $\xi$. This is repeated iteratively until the penalty reaches $\epsilon$. The penalty is then removed in the final round of stabilization to ensure we find an optimal solution to the master problem.

\subsection{Separation and handling of consistency cuts}

Rather than adding all consistency cuts at the beginning, as shown in $(\mathsf{RMP})$, we add them dynamically in rounds. After the DW relaxation is solved using column generation, we look for consistency cuts that are violated by the current solution. We then add these cuts to the DW relaxation and update the subproblems accordingly. Column generation is then resumed. This continues in \emph{cut rounds} until no more cuts can be added. A consistency cut is denoted by a triple $(i,j,q)$ where $i$ and $j$ are the blocks and $q$ is the pattern.

To separate the consistency cuts in an efficient manner it is convenient
to define data structures $\Phi_{ij}$ for $i,j\in\{1,\ldots,k\},i<j$.
Each element $\Phi{}_{ij}$ is a dictionary that maps from a pattern
$q$ between subproblem $i$ and $j$ to a pair $(v_{qi},v_{qj})$
where $v_{ql}=\sum_{p\in\mathcal{\bar{P}}_{l}(q)}\lambda_{lp}^{*}$.
That is, the first (resp. second) component of the pair stores the
total weight, in the current solution, of the extreme points from
block $i$ (resp $j$) that matches pattern $q$. Only pairs where
at least one component is non-zero are stored. Once this data structure
has been populated it is a matter of iterating through all its elements:
whenever we encounter a pair where the first and second element differ
then we have found a violated inequality. To limit the number of inequalities we can choose to only return the ones where the difference between
the elements in the pair is larger than a certain threshold $\delta$.

Algorithm \ref{alg:CDS} below shows how $(\Phi_{ij})$ can be constructed
in an efficient way (similar to standard matrix notation we use parentheses
around $\Phi{}_{ij}$ to indicate that we address all the dictionaries,
and not the single dictionary given by $i$ and $j$). The algorithm takes $(\lambda_{ip}^*)$ the solution to current RMP, $(\mathcal{X}_{ij})$ the variables shared between each pair of blocks, and $(\bar{x}_{ip})$ the extreme points in the RMP as inputs and returns $(\Phi_{ij})$. Lines \ref{fori}-\ref{onlynonzero}
scan all the variables in $(\mathsf{RMP})$, looking for the non-zero variables
that can potentially lead to updates of $(\Phi_{ij})$. Once a non-zero
$\lambda$-variable for block $i$ has been found we go through all
other blocks $j$ that share variables with block $i$ (lines \ref{forj}-\ref{checkoverlap}).
In line \ref{extractQ} we extract the pattern $q$ that corresponds
to a specific extreme point and a specific pair of blocks, in line
\ref{lookupQij} we look up this pattern in $\Phi_{ij}$. If the pattern
does not exist in $\Phi_{ij}$ then the \emph{find} function returns
the pair $(0,0)$. In line \ref{updateQij} we update the weights
stored in the pair and insert it back into position $q$. Lines \ref{lookupQij2}
and \ref{updateQij2} are similar but handle the case when $i>j$.

We now turn to the computational complexity of the algorithm. Let
$\eta$ be the number of $\lambda$-variables in the restricted master
problem and let $\bar{\eta}$ be the number of $\lambda$-variables
in the restricted master problem that take a non-zero value, the computational
complexity of the algorithm is then $O(k^{2}+\eta+\bar{\eta}kn)$.
The initial $k^{2}$ comes from initializing the $O(k^{2})$ elements
in $(\Phi_{ij})$ (not shown in Algorithm \ref{alg:CDS}) the $\eta$
term comes from lines \ref{fori} and \ref{forp} that run through
each variable in the restricted master problem. The last $\bar{\eta}kn$
term comes from that for each variable in the restricted master problem
that takes a non-zero value we have to execute the lines \ref{forj}
to \ref{endforj}. The check in line \ref{checkoverlap} can be done
in constant time and line \ref{extractQ} can be done in $O(n)$.
Since $\Phi_{ij}$, for a specific block-pair $i,j$, in the worst
case can contain $2^{n}$ elements (this is a very pessimistic bound)
the find and insert operations can be done in $O(n)$ using an implementation
based on, for example, red-back trees (see \citep{cormen2009introduction}).

Once $(\Phi_{ij})$ has been computed it is trivial to find the violated
inequalities. For each $i,j\in\{1,\ldots,k\}$ where $i<j$, one goes through
all the pairs $(v_{1},v_{2})$ in $\Phi_{ij}$. If $v_{1}\neq v_{2}$
then the pattern $q$, that corresponds to the given pair, together
with the block indices $i$ and $j$ identifies a violated inequality. We refer to the complete process of generating $(\Phi_{ij})$ and finding the violated inequalities as the \emph{cut separation algorithm}. All the data structures in $(\Phi_{ij})$ contains at most $\bar{\eta}k$
pairs in total since we at most add $k$ new pairs to the data structure
every time we enter lines \ref{forj} to \ref{endforj}. The running
time of the entire cut separation algorithm is therefore dominated by
the time to construct the $\Phi_{ij}$ sets and the total running
time of the algorithm is $O(k^{2}+\eta+\bar{\eta}kn)$. The algorithm
is therefore polynomial in its inputs, since it takes the current
RMP solution as input, but not necessarily in the size of the original
MIP that we are solving. The problem is, that the Dantzig-Wolfe reformulated
model may contain a number of variables that are exponential in the
size of the original problem. Even though column generation typically
only generates a small subset of these variables we cannot guarantee
that the number of variables generated (and thereby $\text{\ensuremath{\eta}}$)
does not grow exponentially with the input size of the original problem.
In practice, we expect that solving the linear programming formulation
of the restricted master problem, in general, is going to be more
time consuming than running the cut separation algorithm (since the running
time of the LP solver also depends on $\eta$) . The computational
results in Section \ref{CompResults} support this conjecture.

\begin{algorithm} 
$\mathbf{compute\_Phi}((\lambda_{ip}^*), (\mathcal{X}_{ij}), (\bar{x}_{ip}))$
\begin{algorithmic}[1]
\FOR {$i\in\{1,\ldots,k\}$} \label{fori} 
\FOR {$p\in \bar{\mathcal{P}}_i$} \label{forp}
\IF {$\lambda_{ip}^* > 0$} \label{onlynonzero}
\FOR {$j\in\{1,\ldots,k\},i\neq j$} \label{forj}
\IF {$\mathcal{X}_{ij} \neq \emptyset$} \label{checkoverlap}
\STATE $q=\mathbf{extract\_pattern}(\bar{x}_{ip}, \mathcal{X}_{ij})$ \label{extractQ}
\IF{$i < j$}
\STATE $(v_1,v_2)$ = find($\Phi_{ij}, q$)\label{lookupQij} 
\STATE $\mathbf{insert}(\Phi_{ij},q,(v_1+\lambda_{ip}^*,v_2))$\label{updateQij}
\ELSE
\STATE $(v_1,v_2)$ = find($\Phi_{ji}, q$)\label{lookupQij2}
\STATE $\mathbf{insert}(\Phi_{ji},q,(v_1,v_2+\lambda_{ip}^*))$\label{updateQij2}
\ENDIF
\ENDIF
\ENDFOR \label{endforj}
\ENDIF
\ENDFOR
\ENDFOR
\RETURN $(\Phi_{ij})$
\end{algorithmic}
\caption{Constructs ($\Phi_{ij}$) which can then be used directly to separate all violated consistency cuts.\label{alg:CDS}}
\end{algorithm}

\noindent For each cut $(i,j,q)$ we need to add a variable linking constraint to the restricted master problem; we refer to the dual of this constraint as $\mu_{q}$. Subproblems $i$ and $j$ must be modified by adding the variable $y_q$ and the constraints (\ref{eq:sharing_y1}) and (\ref{eq:sharing_y2}). Furthermore, the subproblem objectives (\ref{eq:SUB1_obj}) are modified by adding the term $\mu_{q}y_{q}$. In our implementation, the $\mu_{q}$ dual values are not stabilized like the other dual values in $(\mathsf{RMP})$. One could experiment with stabilizing them like the others, as it might speed up convergence.

\subsection{Example}

We briefly return to the example (\ref{eq:Ex_obj})-(\ref{eq:Ex_dom}) to show what the linking constraint separation would look like. The possible patterns between subproblem 1 and 2 are: $\{(x_2,0),(x_3,0)\}$, $\{(x_2,1),(x_3,1)\}$, $\{(x_2,0),(x_3,1)\}$, and $\{(x_2,1),(x_3,0)\}$. The first (resp. last) two patterns appear in the extreme points of subproblem 1 (resp. 2) both with a weight of 0.5, i.e. $\Phi_{1,2}$ holds an entry for all four patterns, the first two are given by $(0.5, 0)$ and the last two by $(0, 0.5)$. Running through the elements in $\Phi_{1,2}$ we find that all of them correspond to a violated cut. The next step would be to add these to the problem and resolve it. As seen in Section \ref{subsec:example}, adding just the cut on the $\{(x_2,0),(x_3,0)\}$ pattern is actually enough to reach optimality for this example.

\subsection{Branching}

If the conditions of Theorem \ref{Theorem} applies there is never a need to branch. However, in this work we also test the consistency cuts on problems that do not satisfy all of the conditions, and hence branching will sometimes be necessary. The branching is done on the original variables of the problem. We can easily convert the $\lambda$-variables back to the original variables using (\ref{eq:VS2_set2}). When we branch on a variable, we change the bounds of that variable in the subproblem. We branch on the variable that we assess would have the highest impact on the bound, if it were forced to be integer (branched on), and create two new nodes. Consider a variable $x$ with objective value $c$. We define 

\begin{singlespace}
\[
f = c\cdot min(\ceil{x}-x, x-\floor{x})
\]
\end{singlespace}

\noindent We branch on the variable with the highest value of $f$. Preliminary tests showed that this branching rule worked better than other simpler rules, e.g. the most fractional variable. After branching, we transfer all columns from the parent node to each of the branches, in order to warm start the column generation part of the algorithm. Before adding the columns, we filter away those that do not satisfy the branch. We also transfer all consistency cuts to the branches so these do not need to be generated again. If a node in the branch-and-bound tree is infeasible given the columns from the master, we use Farkas pricing \citep{achterberg2009scip} to try to generate feasible columns. If that is not possible, the node is, in fact, infeasible and can be pruned.

\subsection{Initial solution}

In general, the convergence of column generation is sensitive to the initial solution that the master problem is provided with. The column generation procedure is initialized with a feasible solution that is obtained by running CPLEX on the original problem for a specified amount of time. An initial solution is not strictly necessary. The motivation for including it is to avoid a sequence of Farkas pricing steps to obtain feasibility. If no solution is found within the time limit, then Farkas pricing is used.


\section{Computational results}\label{CompResults}

To show the potential of the consistency cuts we have performed two separate tests. Firstly, the cuts were tested on the TKP. As explained in Section \ref{TKP}, the TKP has the chain structure sufficient for the consistency cuts to ensure optimal solutions without branching. Secondly, a broader test was performed, inspired by the test performed in \citep{bergner2015automatic}. In this test, a subset of the instances from MIPLIB2017 are automatically decomposed and solved using the branch-and-cut-and-price framework. Here there is no guarantee that the consistency cuts will ensure optimality without branching since the conditions of Theorem \ref{Theorem} are not necessarily satisfied. Recall that the algorithm is controlled by several parameters. Preliminary tests showed that the following setting works well, 
$$(\delta, \tau, \eta, \rho, \xi, \epsilon)= (0.05, 10000, 10, 1, 0.15, 0.00000009)$$
The algorithm's performance could probably be improved further with an in-depth tuning of the parameters, but this is not the focus of this work. CPLEX was given 5000 ticks to find an initial solution.

\subsection{TKP}
\label{subsec:tkp}

A benchmark set of 200 instances of the TKP was used for testing. This test set originates from \citep{caprara2013uncommon} and has become the standard way of benchmarking solution methods for the TKP. The test set is split into two categories of 100 problems each, the I-instances and the U-instances. For a detailed description of how the instances are generated see \citep{caprara2013uncommon}. An interesting observation is that the average LP gap of the I-instances is $14.55\%$ whereas it is only $1.61\%$ for the U-instances. This is rather curious since the U-instances are generally harder to solve than the I-instances. However, this is apparently not because of weaker LP bounds.

In the following, the branch-and-cut-and-price algorithm with cuts \textbf{CP} and without cuts \textbf{BP} are benchmarked against CPLEX and other solution methods for the TKP. \citep{caprara2013uncommon} also proposed a DW reformulation with linking variables for solving the TKP. Column generation was used to solve the DW relaxation and a commercial solver was used to solve the subproblems. When comparing results later, we will refer to this method as \textbf{CFM}. \citep{caprara2016solving} uses the same decomposition as \citep{caprara2013uncommon} but always splits the original problem into two evenly sized sets of constraints. This creates two smaller TKP problems that can then be solved recursively in the same way. The authors use 3-4 levels of recursion and solve the subproblems in the final level with a MIP solver. By dividing the problem into two large blocks the bounds are much better than if many smaller subproblems were used; however, unlike when using DW relaxation with consistency cuts, the root node is not guaranteed to be optimal. The paper only benchmarks on a subset of the I-instances and is hence not included in this section. A comparison to our solution methods can, however, be found in Appendix A. The running times of our methods are significantly lower, but it is unclear whether this is due to stabilization, improved hardware, or algorithmic differences. Other solution methods have been tested; \citep{gschwind2017stabilized} proposed a stabilized version of the column generation from \citep{caprara2013uncommon}, and \citep{clautiaux2019dynamic} uses dynamic programming. We will refer to these as \textbf{GI} and \textbf{CDG} respectively.

Both \citep{caprara2013uncommon} and \citep{gschwind2017stabilized} showed that the block size, i.e. number of constraints in each block, has a huge impact on the solution time. This prompted us to also test different block sizes, with and without consistency cuts. With consistency cuts the sizes \{16,32,64,96,128\} were tested, and without, we tested \{16,32,64,96,128,192,256,384\}. The best performing size proved to be 32 with cuts, and 256 without. The performance of the method proved much less sensitive to choosing the best block size when using cuts compared to not using cuts. The extended results can be found in Appendix B, which also shows the average amount of time spent solving the subproblems and the master problem as well as the time spent separating cuts. In general the subproblems take up almost all the the solution time, the time spent separating the cuts is very insignificant in comparison.

Table \ref{CutRounds} shows the number of instances solved to integer optimality after $n$ cut rounds when using block size 32 (the best performing). The U-instances need a few more rounds than the I-instances; however, almost all of the instances are solved after just 5 rounds. The remaining 5 instances, which are also the ones that take the most time to solve, are solved within 8 rounds.

Table \ref{CompareResults} gives an overview of our results and the results of the other methods, that have been used to solve the 200 test instances in the past. It also shows the performance of CPLEX 12.8, referred to as \textbf{CPLEX}. The table does not provide a completely fair comparison since some of the methods use older versions of CPLEX, older hardware, fewer threads, or less time than others. The hardware used for \textbf{CPLEX}, \textbf{CP}, and \textbf{BP} was 4 threads on a XeonE5\_2680v2 with 2.8 GHz and 16 GB RAM, and version 12.8 of CPLEX. \textbf{CFM} used CPLEX 12.1 and a single thread on an INTEL Core2 Duo E6550 with 2.33 GHz and 2 GB RAM. \textbf{GI} used CPLEX 12.5 and one thread on an Intel(R) Core(TM) i7-2600 with 3.4 GHz and 16 GB RAM. Finally, \textbf{CDG} used CPLEX 12.7 and 6 threads on two Dodeca-core Haswell INTEL XeonE5\_2680v3 with 2.5 GHz and 32 GB. The table shows \textbf{CDG} with both a 1-hour and 3-hour time limit. The block sizes used for \textbf{CP} and \textbf{BP} were 32 and 256 as these were, respectively, the best performing. The results for \textbf{CFM} and \textbf{GI} were obtained with a less strict criterion on the block size. For these the best performing block sizes were chosen on a per-instance basis from the sizes \{1,2,4,8,16,32,64,128,256\} for \textbf{CFM} and \{16,32,64,96,128\} for \textbf{GI}. N/A in the table stands for not available as this data was not presented in the corresponding papers.

\textbf{CFM} solves the fewest instances but it is also the oldest algorithm. \textbf{GI} is a stabilized version of \textbf{CFM} and performs better. It is unclear how much of this is due to improvements to CPLEX and hardware, and how much is actually from the stabilization. \textbf{BP} is essentially our own version of \textbf{GI}. Clearly, the results are better, but the comparison is not fair. However, \textbf{BP} and \textbf{CPLEX} can be compared and here \textbf{BP} is slightly superior. \textbf{CDG} has shown the best results on the 200 instances outside of this work. However, to beat CPLEX, they need more than a 1-hour time limit. They show in their paper that they will indeed outperform CPLEX when both are given a 3-hour time limit. Finally, \textbf{CP} shows the best results so far, outperforming \textbf{CPLEX}, \textbf{BP}, and \textbf{CDG}. 

Figure \ref{Comparison} plots the number of instances solved after x hours for the methods \textbf{CP}, \textbf{BP}, \textbf{CPLEX}, and \textbf{CDG}. In order to make this comparison as fair as possible, all methods use a single thread; however, as explained earlier, for the \textbf{CDG} results CPLEX 12.7 was used as opposed to 12.8 for the others. On the contrary, \textbf{CDG} used slightly newer hardware as well as double the amount of RAM. All four methods were given a time limit of 3 hours. \textbf{CDG} surpasses \textbf{CPLEX} after 1.2 hours. It also seems that, if given more time, \textbf{CDG} would perform better than \textbf{BP}; \textbf{CDG} is worse than \textbf{BP} until around the 3-hour limit. Most notably, the plot emphasizes how much faster \textbf{CP} is than the other methods. \textbf{CP} finishes each of the 200 instances well within the 3-hour limit, spending 5080 seconds on the hardest one. Furthermore, after just 970 seconds (less than a tenth of the time limit) \textbf{CP} has solved more instances than any of the other three methods manage to solve within the full 10800 seconds limit.

It is important to emphasize that the superior performance of \textbf{CP} is not due to the inclusion of implementation aspects such as stabilization, subproblem management, or starting from a feasible initial solution. The implementations of \textbf{CP} and \textbf{BP} are identical with the exception that consistency cuts are applied to obtain integrality with the former, while branching is used in the latter. The performance improvement is solely due to the inclusion of the consistency cuts.

\begin{singlespace}
\begin{table}
\centering
\begin{tabular}{ccccccccccc}
\# of cut rounds & & 0 & 1 & 2 & 3 & 4 & 5 & 6 & 7 & 8\tabularnewline
\hline
\multirow{2}{*}{\# of optimal} & I & 6 & 75 & 99 & 100 & 100 & 100 & 100 & 100 & 100\tabularnewline
 & U & 17 & 44 & 60 & 78 & 88 & 95 & 97 & 99 & 100\tabularnewline
\end{tabular}
\caption{Number of cut rounds needed with block size 32 for integer optimality of the I and U-instances.}
\label{CutRounds}
\end{table}
\end{singlespace}

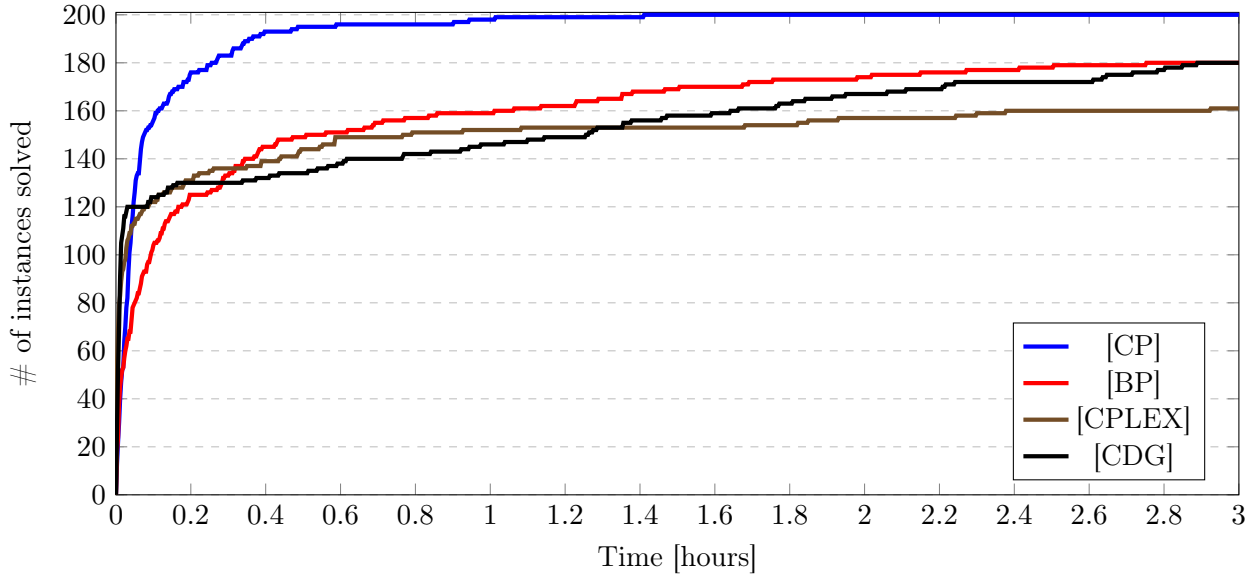
\begin{figure}
    \begin{tikzpicture}
\begin{axis}[
    no markers,
    every axis plot/.append style={ultra thick},
    xlabel={Time [hours]},
    ylabel={\# of instances solved},
    height=8cm,
    width=\textwidth,
    xmin=0, xmax=3,
    ymin=0, ymax=201,
    xtick={0,0.2,0.4,0.6,0.8,1,1.2,1.4,1.6,1.8,2,2.2,2.4,2.6,2.8,3},
    ytick={0,20,40,60,80,100,120,140,160,180,200},
    legend pos=north west,
    ymajorgrids=true,
    grid style=dashed,
    legend pos=south east
]

\addplot table [x=hours, y=sharing, col sep=space] {data.dat};
\addplot table [x=hours, y=bnb, col sep=space] {data.dat};
\addplot table [x=hours, y=cplex, col sep=space] {data.dat};
\addplot table [x=hours, y=cdg, col sep=space] {data.dat};
\addlegendentry{[CP]}
\addlegendentry{[BP]}
\addlegendentry{[CPLEX]}
\addlegendentry{[CDG]}
 
\end{axis}
\end{tikzpicture}
    \caption{Plot of time vs number of instances solved for the methods \textbf{CP}, \textbf{BP}, \textbf{CPLEX}, and \textbf{CDG} when all given a single thread and a 3-hour time limit.}
    \label{Comparison}
\end{figure}

\begin{singlespace}
\begin{table}
\centering
\begin{tabular}{lrrrrr}
\hline 
Method & \# of threads & Time limit [h] & \#Optimal & Avg. gap [\%] & Avg. time [s]\tabularnewline 
\hline 
\textbf{CFM} & 1 & 1 & 138 & 0.07 & N/A \tabularnewline

\textbf{GI} & 1 & 1 & 154 & 0.02 & N/A \tabularnewline

\textbf{CDG} & 6 & 1 & 161 & N/A & N/A \tabularnewline

\textbf{CDG} & 6 & 3 & 194 & N/A & N/A \tabularnewline 

\textbf{CPLEX} & 4 & 1 & 173 & 0.02 & 674\tabularnewline

\textbf{CP} & 4 & 1 & 200 & 0.00 & 201\tabularnewline

\textbf{BP} & 4 & 1 & 179 & <0.01 & 833\tabularnewline



\hline
\end{tabular}
\caption{Comparison of historical results, the performance of our methods and CPLEX on the test set.}
\label{CompareResults}
\end{table}
\end{singlespace}

\subsection{MIPLIB2017}

Solving generic MIPs using DW reformulation is an interesting research direction, that still is in its infancy. Noteworthy contributions include \citep{lubbecke2015primal, bergner2015automatic, wang2013computational}. Currently, solution methods that automatically detect the decomposition structure cannot compete with state-of-the-art branch-and-cut solvers.
The purpose of this test is to show that the consistency cuts can improve the DW relaxation bounds of generic MIPs. We only test the consistency cuts on shared binary variables. We do not binarize any shared integer variables. For this test, we limit ourselves to instances with a density between $0.05$ and $5\%$, at most $20,000$ non-zero elements, and at least $20\%$ of variables must be integer. These characteristics describe instances that are good for DW reformulation, according to \citep{bergner2015automatic}. Using MIPLIB2017 and the listed filters, this leaves us with 171 instances (Notice that in \citep{bergner2015automatic} MIPLIB2003 and MIPLIB2010 were used instead). We also filter away infeasible instances. Decompositions are generated automatically using hypergraph partitioning, specifically \texttt{hMETIS} \citep{hmetis}. We define each problem as a hypergraph using the \emph{Row-Column-Net Hypergraph} method described in \citep{bergner2015automatic}. In this method, each non-zero element in the constraint matrix is a vertex in the graph. For each variable (column) in the problem, we define a hyperedge spanning all vertices corresponding to non-zero elements in that column. Edges corresponding to continuous variables are given a weight of 1 and integer/binary variables a weight of 2. This means that \texttt{hMETIS} would rather allow continuous linking variables than integer or binary. Similarly, for each constraint (row), we create a hyperedge spanning the non-zero elements in the row. These are given a weight of 5. Given a number of blocks, \texttt{hMETIS} separates the vertices into that many components. Following  \citep{bergner2015automatic}, we add $20\%$ more dummy vertices that are not included in any hyperedges. Per default, \texttt{hMETIS} tries to make evenly sized components; however, the dummy vertices allow for some variance in size. From the partition, we generate the corresponding decomposition. Each component of the partition corresponds to a block. If all non-zero elements of a constraint are in the same component, that constraint is put in the corresponding block. Constraints that do not fit in any block are considered coupling constraints. For each instance, we generate five decompositions with sizes 2, 5, 10, 20, and 50 blocks. \citep{bergner2015automatic} examined sizes ranging from 2 to 10; however, we include decompositions with many smaller blocks as well because this worked well for the TKP experiments. All tests are given a 1-hour time limit using the same hardware and same version of CPLEX as specified in Section~\ref{subsec:tkp}.

Table \ref{MIPLIB1} shows an overview of the MIPLIB results. It categorizes the instances according to different criteria. The instances of each category are a subset of the instances in the categories to the left of that category. The \emph{solved} category includes all instances where at least one of the tested decompositions was solved in the root node within the 1-hour time limit (no consistency cuts). The framework that we use to solve the instances does not handle unbounded subproblems; there are 37 instances where all five decompositions have unbounded blocks, these count as not solved. Out of the remaining instances, two of them are not solved because the decomposition generated by \texttt{hMETIS} is empty, i.e. all constraints are coupling constraints. The category \emph{not optimal in root} contains instances where one or more decompositions were solved but optimality is not proven. In other words, these are the instances where the bound of at least one of the decompositions could be improved. The instances with \emph{cuts applicable} satisfy all of the above, but also have the property that there are at least one pair of blocks that share at least two binary variables. This essentially means that at least one consistency cut could potentially be applied. The category \emph{cuts applied} counts the instances where cuts were applied for at least one of the decompositions. In the \emph{LB improved} category, the LB of at least one of the five decompositions was improved with consistency cuts. Finally, in the category \emph{best LB improved}, the best lower bound out of all five decompositions was found using consistency cuts. The table shows that out of the 62 instances where cuts can be applied, the cuts help on at least one of the tested decompositions for 20 of them and improves the best lower bound on 10 of them. For most of the instances where the cuts were applied and did not improve the solution, the method ran out of time.

We took a closer look at the instances from category \emph{best LB improved} in Table \ref{MIPLIB2}. For each instance, the results are given for each of the five decompositions (notice that \texttt{hMETIS} sometimes makes partitions that do not use the full number of allowed blocks). The best known upper bounds shown are taken directly from the MIPLIB2017 website. For each instance, the lower bound and time to find it are shown with and without consistency cuts. A "-" means that the algorithm ran out of time. Lower bounds are reported even when the algorithm has run out of time, before finishing the root node computation. In this case, the reported number is the best lower bound observed before running out of time. Lower bounds are computed "continuously" using the technique described in \citep{lasdon1970ptimization}, section 3.7. The best found lower bound for each instance is marked in bold. The bounds are compared to the LP relaxation as well as the best solution found by CPLEX. All of the 10 problems are easily solved by CPLEX, showing that the proposed methodology still is not competitive with a state-of-the-art solver on generic MIPs. The table also shows the number of cut rounds used and the total number of cuts added. Finally, the gap closed is shown, calculated with the formula $\frac{LB(cuts)-LB(no cuts)}{UB-LB(no cuts)}$. A 100\% improvement means that the entire gap before cuts was closed when adding the cuts. For 3 of the instances, the gap is closed completely using the cuts. The last instance, \emph{qnet1\_0}, is interesting because the best bound is not found with the 2 block decomposition like the others but rather the 10 block decomposition. It should be noted that generating automatic decompositions does not guarantee that we obtain the chain structure described earlier. Appendix C shows the results for the remaining instances in the \emph{LB improved} category. In this test, we have shown that even without the chain structure, the cuts can improve the lower bound of generic MIPs.

\begin{singlespace}
\begin{table}
\centering
\begin{tabular}{ccccccc}
\hline 
\multirow{2}{*}{total} & \multirow{2}{*}{solved} & not optimal & cuts & cuts & \multirow{2}{*}{LB improved} & best LB\tabularnewline
 & & in root & applicable & applied & & improved\tabularnewline
\hline 
171 & 113 & 106 & 62 & 52 & 20 & 10 \tabularnewline
\hline
\end{tabular}
\caption{Overview of MIPLIB results.}
\label{MIPLIB1}
\end{table}
\end{singlespace}

\begin{small}
\begin{landscape}
\begin{singlespace}
\setlength{\tabcolsep}{2pt}
\centering
\begin{longtable}{cccccccccccccccccccc}
\hline 
\multirow{2}{*}{instance} & best known &  & \multicolumn{3}{c}{CPLEX} &  & \# of &  & \multicolumn{2}{c}{DW no consistency cuts} &  & \multicolumn{2}{c}{DW consistency cuts} &  & cut &  & cuts &  & gap\tabularnewline
\cline{4-6} \cline{10-11} \cline{13-14} 
 & solution &  & LP relax & LB & time [s] &  & blocks &  & LB & time [s] &  & LB & time [s] &  & rounds &  & added &  & improved [\%]\tabularnewline
\hline 
\multirow{5}{*}{neos-831188} & \multirow{5}{*}{2.61} &  & \multirow{5}{*}{1.77} & \multirow{5}{*}{2.61} & \multirow{5}{*}{17.7} &  & 2 &  & 2.37 & 1092.5 &  & 2.45 & - &  & 12 &  & 210 &  & 35.69\tabularnewline
 &  &  &  &  &  &  & 5 &  & 2.04 & 649.9 &  & 2.06 & - &  & 2 &  & 1019 &  & 3.71\tabularnewline
 &  &  &  &  &  &  & 10 &  & 1.87 & 585.0 &  & 1.87 & - &  & 1 &  & 2032 &  & 0\tabularnewline
 &  &  &  &  &  &  & 20 &  & 1.82 & 715.2 &  & 1.82 & - &  & 1 &  & 7684 &  & 0\tabularnewline
 &  &  &  &  &  &  & 50 &  & 1.77 & 752.1 &  & 1.77 & - &  & 1 &  & 7965 &  & 0\tabularnewline
\hline 
\multirow{5}{*}{toll-like} & \multirow{5}{*}{610} &  & \multirow{5}{*}{0.00} & \multirow{5}{*}{610} & \multirow{5}{*}{272.8} &  & 2 &  & 584 & 363.9 &  & 585 & - &  & 8 &  & 55 &  & 3.85\tabularnewline
 &  &  &  &  &  &  & 5 &  & 574 & 129.2 &  & 580.71 & - &  & 4 &  & 318 &  & 18.61\tabularnewline
 &  &  &  &  &  &  & 10 &  & 547 & 69.1 &  & 572 & - &  & 3 &  & 586 &  & 39.68\tabularnewline
 &  &  &  &  &  &  & 20 &  & 525 & 60.2 &  & 557.70 & - &  & 3 &  & 1354 &  & 38.46\tabularnewline
 &  &  &  &  &  &  & 50 &  & 472 & 46.0 &  & 565.11 & - &  & 3 &  & 2175 &  & 67.47\tabularnewline
\hline 
\multirow{5}{*}{enlight8} & \multirow{5}{*}{27} &  & \multirow{5}{*}{0.00} & \multirow{5}{*}{27} & \multirow{5}{*}{0.0} &  & 2 &  & 14 & 28.3 &  & 27 & 72.9 &  & 6 &  & 96 &  & 100\tabularnewline
 &  &  &  &  &  &  & 5 &  & 8.17 & 26.6 &  & 27 & 254.8 &  & 7 &  & 661 &  & 100\tabularnewline
 &  &  &  &  &  &  & 10 &  & 7.38 & 26.2 &  & 20.22 & 573.6 &  & 8 &  & 734 &  & 65.46\tabularnewline
 &  &  &  &  &  &  & 20 &  & 2.20 & 26.3 &  & 10.51 & 461.8 &  & 6 &  & 867 &  & 33.48\tabularnewline
 &  &  &  &  &  &  & 50 &  & 1.77 & 26.2 &  & 6.30 & 90.3 &  & 7 &  & 579 &  & 17.94\tabularnewline
\hline 
\multirow{5}{*}{enlight\_hard} & \multirow{5}{*}{37} &  & \multirow{5}{*}{0.00} & \multirow{5}{*}{37} & \multirow{5}{*}{0.0} &  & 2 &  & 34.11 & 40.0 &  & 37 & 68.9 &  & 2 &  & 44 &  & 100\tabularnewline
 &  &  &  &  &  &  & 5 &  & 27.20 & 31.6 &  & 37 & 635.8 &  & 8 &  & 917 &  & 100\tabularnewline
 &  &  &  &  &  &  & 10 &  & 23.28 & 28.6 &  & 36.21 & 3612.5 &  & 9 &  & 1746 &  & 94.25\tabularnewline
 &  &  &  &  &  &  & 20 &  & 20.03 & 27.3 &  & 31.88 & 2099.1 &  & 6 &  & 1263 &  & 69.84\tabularnewline
 &  &  &  &  &  &  & 50 &  & 17.19 & 27.9 &  & 24.68 & 452.5 &  & 4 &  & 1095 &  & 37.81\tabularnewline
\hline 
\multirow{5}{*}{neos-2652786-brda} & \multirow{5}{*}{4.79} &  & \multirow{5}{*}{0.00} & \multirow{5}{*}{4.79} & \multirow{5}{*}{143.9} &  & 2 &  & 1.67 & 153.1 &  & 1.81 & 237.5 &  & 1 &  & 6 &  & 4.39\tabularnewline
 &  &  &  &  &  &  & 5 &  & 0.26 & 143.4 &  & 0.28 & 288.4 &  & 2 &  & 39 &  & 0.37\tabularnewline
 &  &  &  &  &  &  & 10 &  & 0.17 & 95.9 &  & 0.18 & 131.1 &  & 1 &  & 27 &  & 0.26\tabularnewline
 &  &  &  &  &  &  & 20 &  & 0 & 50.2 &  & 2.14 & 514.0 &  & 9 &  & 915 &  & 0\tabularnewline
 &  &  &  &  &  &  & 50 &  & 0 & 36.4 &  & 0 & 42.7 &  & 0 &  & 0 &  & 0\tabularnewline
\hline 
\multirow{5}{*}{22433} & \multirow{5}{*}{21477} &  & \multirow{5}{*}{21240.53} & \multirow{5}{*}{21477} & \multirow{5}{*}{0.1} &  & 2 &  & 21472.28 & 479.4 &  & 21473.22 & - &  & 12 &  & 645 &  & 19.94\tabularnewline
 &  &  &  &  &  &  & 5 &  & 21332.85 & 490.6 &  & 21332.85 & - &  & 3 &  & 2390 &  & 0\tabularnewline
 &  &  &  &  &  &  & 10 &  & 21246.43 & 325.3 &  & 21246.43 & - &  & 1 &  & 2280 &  & 0\tabularnewline
 &  &  &  &  &  &  & 20 &  & 21244.06 & 108.2 &  & 21244.06 & - &  & 1 &  & 2449 &  & 0\tabularnewline
 &  &  &  &  &  &  & 48 &  & 21240.53 & 49.0 &  & 21240.53 & - &  & 1 &  & 1476 &  & 0\tabularnewline
\hline 
\multirow{5}{*}{fiber} & \multirow{5}{*}{405935.18} &  & \multirow{5}{*}{156082.52} & \multirow{5}{*}{405935.18} & \multirow{5}{*}{0.3} &  & 2 &  & 402702.35 & 54.6 &  & 405935.18 & 169.7 &  & 11 &  & 94 &  & 100\tabularnewline
 &  &  &  &  &  &  & 5 &  & 392217.89 & 96.6 &  & 395650.62 & - &  & 18 &  & 627 &  & 25.02\tabularnewline
 &  &  &  &  &  &  & 10 &  & 390493.82 & 84.9 &  & 403591.41 & - &  & 6 &  & 722 &  & 84.82\tabularnewline
 &  &  &  &  &  &  & 20 &  & 390493.82 & 83.5 &  & 394889.57 & - &  & 8 &  & 492 &  & 28.47\tabularnewline
 &  &  &  &  &  &  & 49 &  & 233791.93 & 59.5 &  & 233791.93 & 45.8 &  & 1 &  & 56 &  & 0\tabularnewline
\hline 
\multirow{5}{*}{neos-3611689-kaihu} & \multirow{5}{*}{119} &  & \multirow{5}{*}{100.88} & \multirow{5}{*}{119} & \multirow{5}{*}{0.5} &  & 2 &  & 114.93 & 37.8 &  & 115.09 & 56.7 &  & 2 &  & 22 &  & 3.81\tabularnewline
 &  &  &  &  &  &  & 5 &  & 113.65 & 37.5 &  & 113.75 & 60.3 &  & 3 &  & 82 &  & 1.83\tabularnewline
 &  &  &  &  &  &  & 10 &  & 111.66 & 32.5 &  & 111.72 & 31.9 &  & 1 &  & 61 &  & 0.85\tabularnewline
 &  &  &  &  &  &  & 20 &  & 108.91 & 33.1 &  & 108.91 & 29.6 &  & 1 &  & 14 &  & 0\tabularnewline
 &  &  &  &  &  &  & 50 &  & 105.29 & 32.9 &  & 105.29 & 42.0 &  & 1 &  & 22 &  & 0\tabularnewline
\hline 
\multirow{5}{*}{rout} & \multirow{5}{*}{1077.56} &  & \multirow{5}{*}{981.86} & \multirow{5}{*}{1077.47} & \multirow{5}{*}{4.8} &  & 2 &  & \multicolumn{2}{c}{Unbounded sup problem} &  &  &  &  & \tabularnewline
 &  &  &  &  &  &  & 5 &  & 1036.85 & 128.2 &  & 1037.81 & 3366.1 &  & 10 &  & 416 &  & 2.36\tabularnewline
 &  &  &  &  &  &  & 10 &  & \multicolumn{2}{c}{Unbounded sup problem} &  &  &  &  & \tabularnewline
 &  &  &  &  &  &  & 20 &  & \multicolumn{2}{c}{Unbounded sup problem} &  &  &  &  & \tabularnewline
 &  &  &  &  &  &  & 50 &  & 981.86 & 30.3 &  & 981.86 & 30.4 &  & 0 &  & 0 &  & 0\tabularnewline
\hline 
\multirow{5}{*}{qnet1\_o} & \multirow{5}{*}{16029.69} &  & \multirow{5}{*}{12095.57} & \multirow{5}{*}{16029.19} & \multirow{5}{*}{0.4} &  & 2 &  & 15837.64 & - &  & 15837.64 & - &  & 0 &  & 0 &  & 0\tabularnewline
 &  &  &  &  &  &  & 5 &  & 15837.64 & 1386.0 &  & 15849.97 & - &  & 2 &  & 58 &  & 6.42\tabularnewline
 &  &  &  &  &  &  & 10 &  & 15837.64 & 1088.5 &  & 16000.49 & - &  & 2 &  & 44 &  & 84.8\tabularnewline
 &  &  &  &  &  &  & 20 &  & 15498.91 & 520.2 &  & 15901.59 & - &  & 2 &  & 238 &  & 75.87\tabularnewline
 &  &  &  &  &  &  & 50 &  & 13963.27 & 33.6 &  & 13963.27 & 33.5 &  & 0 &  & 0 &  & 0\tabularnewline
\hline 
\caption{Extended MIPLIB results for instances where consistency cuts improve the best lower bound} \label{MIPLIB2}
\end{longtable}
\end{singlespace}
\end{landscape}
\end{small}

\section{Conclusion}\label{Conclusion}

In this paper, we studied how to reformulate (mixed) integer programming problems using DW reformulation with linking variables. The paper contains four main results: 1) We propose a new family of valid inequalities to be applied to any DW reformulation of a MIP where two blocks in the reformulation share multiple binary variables. 2) We show that, for DW reformulations of MIPs that satisfy certain properties, it is enough to solve the LP relaxation of the DW reformulation with all consistency cuts to obtain integer solutions. 3) We show that the theory is applicable in practice by applying it to the temporal knapsack problem. This leads to a state-of-the-art algorithm for the problem that is able to solve all instances from the standard benchmark set used in the academic literature. Earlier algorithms that have been proposed in the literature have not succeeded in this. 4) We prove that the cuts can improve DW relaxation bounds on generic MIPs from MIPLIB2017.


The results shown here motivate further research into the consistency cuts. There are two immediately interesting directions to take this research. One could look for other problems that, like the TKP, have the extended chain structure and apply the consistency cuts. Alternatively, one could investigate methods for recognizing the extended chain structure, or structures that closely resemble it, in generic MIPs. This would most likely improve the performance of the cuts in an experiment like the one made here with the instances from MIPLIB2017.


\section*{Acknowledgements}

The authors gratefully acknowledge the financial support of the Danish Council for Independent Research (Grant ID: DFF - 7017-00341). And would like to thank the two anonymous reviewers and the associate editor Prof. Daniel Kuhn, for their constructive comments that have improved the paper. We also thank Prof. Guy Desaulniers, Prof. Stefan Irnich, and Prof. Jesper Larsen for their comments to the final version of the paper. Finally, we thank Prof. Marco L\"{u}bbecke for insightful comments on our research.

\section*{Author Biographies}

\textbf{Jens Vinther Clausen} recently completed his PhD at the Department of Technology, Management and Economics, Technical University of Denmark. His research interests are mathematical programming, exact solutions, decompositions methods, metaheuristics, and machine learning.

\noindent\textbf{Stefan Ropke} is a professor at Department of Technology, Management and Economics, Technical University of Denmark. His research interests include mathematical programming, decomposition algorithms and metaheuristics. His main field of application is transportation, including freight distribution, maritime transport and public transport.

\noindent\textbf{Richard Lusby} is an Associate Professor at the Department of Technology, Management and Economics, Technical University of Denmark. His research interests include integer programming, decomposition algorithms, and matheuristics, and their applications to public transportation and manpower planning problems.

\noindent

\bibliographystyle{plainnat}
\bibliography{refs}

\appendix

\section*{Appendix A}\label{Appen A}

\begin{singlespace}
\begin{table}[H]
\centering
\begin{tabular}{ccccc}
\hline 
instance & recursive DW time {[}s{]} & root gap DW {[}\%{]} & BP time {[}s{]} & CP time {[}s{]}\tabularnewline
\hline 
I41 & 93.5 & 0.018 & 151.2 & 61.1\tabularnewline
I43 & 1107.4 & 0 & 463.8 & 81.1\tabularnewline
I45 & 258.4 & 0 & 409.0 & 106.1\tabularnewline
I47 & 303.2 & 0 & 78.4 & 148.0\tabularnewline
I49 & 2322.1 & 0 & 230.8 & 179.3\tabularnewline
I61 & 190.5 & 0 & 17.0 & 20.3\tabularnewline
I63 & 897.8 & 0 & 84.4 & 29.4\tabularnewline
I65 & 3430.7 & 0.04 & 22.9 & 34.5\tabularnewline
I67 & 2523.9 & 0 & 30.6 & 33.2\tabularnewline
I69 & 1510.7 & 0 & 44.7 & 54.6\tabularnewline
I71 & 69.9 & 0 & 62.0 & 61.5\tabularnewline
I73 & 276.9 & 0 & 86.8 & 80.8\tabularnewline
I75 & 395.5 & 0 & 91.8 & 123.3\tabularnewline
I77 & 949.0 & 0 & 235.9 & 119.0\tabularnewline
I79 & 1056.8 & 0 & 234.7 & 192.7\tabularnewline
I81 & 271.5 & 0 & 453.3 & 95.0\tabularnewline
I83 & 948.4 & 0 & 609.8 & 77.6\tabularnewline
I85 & 784.8 & 0 & 414.4 & 114.4\tabularnewline
I87 & 961.7 & 0 & 164.6 & 151.7\tabularnewline
I89 & 1221.5 & 0 & 276.1 & 230.5\tabularnewline
I91 & 130.9 & 0 & 95.4 & 81.7\tabularnewline
I93 & 390.2 & 0 & 118.4 & 103.1\tabularnewline
I95 & 1161.6 & 0 & 122.9 & 117.2\tabularnewline
I97 & 2869.5 & 0 & 91.0 & 133.9\tabularnewline
I99 & 2424.1 & 0 & 425.2 & 166.1\tabularnewline
\hline 
\end{tabular}
\caption{Performance comparison between the methods from this paper (BP and CP) and the recursive DW reformulation presented in \citep{caprara2016solving}}
\end{table}
\end{singlespace}

\section*{Appendix B}\label{Appen B}

\begin{singlespace}
\begin{table}[H]
\centering
\begin{tabular}{ccrrrrrrrr}
\hline 
\thead{Block \\ size \\ \ } &  & \thead{\# of \\ optimal \\ \ } & \thead{Avg. \\ gap [\%] \\ \ } & \thead{Avg. \\ time {[}s{]} \\ \ } & \thead{Avg. \\ time \\ ($\mathsf{RMP}$) {[}s{]}} & \thead{Avg. \\ time \\ sub {[}s{]}} & \thead{Avg. \# \\ of col \\ gen iters} & \thead{Avg. \# of \\ branch-and \\ -price nodes} & \thead{Avg. \# of \\ columns \\ added}\tabularnewline
\hline 
\multirow{2}{*}{16} & I & 21 & 0.05 & 3029 & 11 & 2984 & 2414 & 108 & 154064\tabularnewline
 & U & 40 & 0.06 & 2347 & 254 & 2054 & 5125 & 87 & 207570\tabularnewline
\hline
\multirow{2}{*}{32} & I & 50 & 0.01 & 2249 & 4 & 2221 & 1438 & 67 & 51365\tabularnewline
 & U & 57 & 0.02 & 1853 & 38 & 1798 & 2980 & 51 & 69044\tabularnewline
\hline
\multirow{2}{*}{64} & I & 77 & <0.01 & 1274 & 1 & 1257 & 689 & 31 & 13183\tabularnewline
 & U & 75 & 0.01 & 1325 & 3 & 1313 & 991 & 19 & 11148\tabularnewline
\hline 
\multirow{2}{*}{96} & I & 75 & <0.01 & 1173 & 1 & 1156 & 513 & 23 & 6623\tabularnewline
 & U & 78 & <0.01 & 1272 & 1 & 1260 & 536 & 10 & 3372\tabularnewline
\hline 
\multirow{2}{*}{128} & I & 93 & <0.01 & 618 & <1 & 604 & 258 & 12 & 2474\tabularnewline
 & U & 79 & <0.01 & 1133 & 1 & 1122 & 225 & 5 & 981\tabularnewline
\hline
\multirow{2}{*}{192} & I & 92 & <0.01 & 774 & <1 & 759 & 197 & 9 & 1435\tabularnewline
 & U & 78 & 0.01 & 1356 & <1 & 1346 & 171 & 4 & 451\tabularnewline
\hline 
\multirow{2}{*}{256} & I & 97 & <0.01 & 474 & <1 & 462 & 123 & 6 & 561\tabularnewline
 & U & 82 & 0.01 & 1190 & <1 & 1181 & 65 & 2 & 102\tabularnewline
\hline 
\multirow{2}{*}{384} & I & 92 & <0.01 & 522 & <1 & 509 & 99 & 4 & 279\tabularnewline
 & U & 61 & - & 2476 & <1 & 2467 & 42 & 1 & 49\tabularnewline
\hline
\end{tabular}
\caption{Performance comparison without consistency cuts for different block sizes when solving the I and U-instances.}
\label{BlockSizeBnP}
\end{table}
\end{singlespace}

\begin{singlespace}
\begin{table}[H]
\centering
\begin{tabular}{ccrrrrrrrrr}
\hline 
\thead{Block \\ size \\ \ } &  & \thead{\# of \\ optimal \\ \ } & \thead{Avg. \\ gap [\%] \\ \ } & \thead{Avg. \\ time {[}s{]} \\ \ } & \thead{Avg. \\ time \\ ($\mathsf{RMP}$) {[}s{]}} & \thead{Avg. \\ time \\ sub {[}s{]}} & \thead{Avg. \\ time cut \\ sep {[}s{]}} & \thead{Avg. \# \\ of col \\ gen iters} & \thead{Avg. \# \\ of cut \\ rounds} & \thead{Avg. \# of \\ columns \\ added}\tabularnewline
\hline 
\multirow{2}{*}{16} & I & 100 & 0.00 & 132 & <1 & 119 & 0.79 & 89 & 1.46 & 1327\tabularnewline
 & U & 98 & <0.01 & 519 & 53 & 454 & 0.48 & 372 & 2.61 & 2904\tabularnewline
\hline
\multirow{2}{*}{32} & I & 100 & 0.00 & 134 & <1 & 122 & 0.52 & 74 & 1.20 & 657\tabularnewline
 & U & 100 & 0.00 & 267 & 3 & 254 & 0.34 & 245 & 2.22 & 1197\tabularnewline
\hline
\multirow{2}{*}{64} & I & 100 & 0.00 & 118 & <1 & 106 & 0.41 & 62 & 0.97 & 326\tabularnewline
 & U & 100 & 0.00 & 340 & 1 & 330 & 0.24 & 166 & 1.46 & 502\tabularnewline
\hline
\multirow{2}{*}{96} & I & 100 & 0.00 & 162 & <1 & 149 & 0.41 & 55 & 0.89 & 221\tabularnewline
 & U & 96 & <0.01 & 645 & 1 & 633 & 0.26 & 147 & 1.26 & 326\tabularnewline
\hline
\multirow{2}{*}{128} & I & 100 & 0.00 & 130 & <1 & 117 & 0.34 & 50 & 0.72 & 164\tabularnewline
 & U & 91 & <0.01 & 776 & <1 & 766 & 0.14 & 110 & 0.60 & 200\tabularnewline
\hline
\end{tabular}
\caption{Performance comparison with consistency cuts for different block sizes when solving the I and U-instances.}
\label{BlockSize}
\end{table}
\end{singlespace}

\section*{Appendix C}\label{Appen C}

\begin{small}
\begin{landscape}
\begin{singlespace}
\setlength{\tabcolsep}{2pt}
\centering
\begin{longtable}{cccccccccccccccccccc}
\hline 
\multirow{2}{*}{instance} & best know &  & \multicolumn{3}{c}{CPLEX} &  & \# of &  & \multicolumn{2}{c}{DW no consistency cuts} &  & \multicolumn{2}{c}{DW consistency cuts} &  & cut &  & cuts &  & gap\tabularnewline
\cline{4-6} \cline{10-11} \cline{13-14} 
 & solution &  & LP relax & LB & time [s] &  & blocks &  & LB & time [s] &  & LB & time [s] &  & rounds &  & added &  & improved [\%]\tabularnewline
\hline 
\multirow{5}{*}{chromaticindex32-8} & \multirow{5}{*}{4} &  & \multirow{5}{*}{3} & \multirow{5}{*}{4} & \multirow{5}{*}{4.7} &  & 2 &  & 4 & 36.5 &  & 4 & 39.7 &  & 0 &  & 0 &  & 0\tabularnewline
 &  &  &  &  &  &  & 5 &  & 4 & 37.6 &  & 4 & 41.3 &  & 0 &  & 0 &  & 0\tabularnewline
 &  &  &  &  &  &  & 10 &  & 4 & 33.6 &  & 4 & 37.2 &  & 0 &  & 0 &  & 0\tabularnewline
 &  &  &  &  &  &  & 20 &  & 4 & 32.7 &  & 4 & 36.1 &  & 0 &  & 0 &  & 0\tabularnewline
 &  &  &  &  &  &  & 50 &  & 3 & 57.0 &  & 3.2 & - &  & 13 &  & 7141 &  & 20\tabularnewline
\hline 
\multirow{5}{*}{supportcase14} & \multirow{5}{*}{288} &  & \multirow{5}{*}{272} & \multirow{5}{*}{288} & \multirow{5}{*}{0.1} &  & 2 &  & 288 & 25.2 &  & 288 & 27.9 &  & 0 &  & 0 &  & 0\tabularnewline
 &  &  &  &  &  &  & 5 &  & 274 & 30.0 &  & 287.65 & - &  & 7 &  & 639 &  & 97.47\tabularnewline
 &  &  &  &  &  &  & 10 &  & 274.67 & 27.6 &  & 280 & 59.8 &  & 3 &  & 142 &  & 40\tabularnewline
 &  &  &  &  &  &  & 20 &  & 274 & 29.4 &  & 278.5 & 91.1 &  & 5 &  & 187 &  & 32.14\tabularnewline
 &  &  &  &  &  &  & 49 &  & 272 & 27.1 &  & 275 & 32.0 &  & 2 &  & 28 &  & 18.75\tabularnewline
\hline 
\multirow{5}{*}{supportcase16} & \multirow{5}{*}{288} &  & \multirow{5}{*}{272} & \multirow{5}{*}{288} & \multirow{5}{*}{0.2} &  & 2 &  & 288 & 27.0 &  & 288 & 30.6 &  & 0 &  & 0 &  & 0\tabularnewline
 &  &  &  &  &  &  & 5 &  & 274 & 31.6 &  & 286.11 & - &  & 9 &  & 694 &  & 86.47\tabularnewline
 &  &  &  &  &  &  & 10 &  & 274.67 & 29.2 &  & 280 & 45.7 &  & 3 &  & 82 &  & 40\tabularnewline
 &  &  &  &  &  &  & 20 &  & 272 & 29.8 &  & 283.33 & 81.2 &  & 6 &  & 168 &  & 70.83\tabularnewline
 &  &  &  &  &  &  & 45 &  & 272 & 28.8 &  & 275 & 35.8 &  & 2 &  & 32 &  & 18.75\tabularnewline
\hline 
\multirow{5}{*}{neos-686190} & \multirow{5}{*}{6730} &  & \multirow{5}{*}{5134.81} & \multirow{5}{*}{6730} & \multirow{5}{*}{19.5} &  & 2 &  & 5545.30 & 841.6 &  & 5545.3 & 843.0 &  & 0 &  & 0 &  & 0\tabularnewline
 &  &  &  &  &  &  & 5 &  & 5545.30 & 137.4 &  & 5545.3 & 137.5 &  & 0 &  & 0 &  & 0\tabularnewline
 &  &  &  &  &  &  & 10 &  & 5215.26 & 309.2 &  & 5301.06 & - &  & 9 &  & 558 &  & 5.66\tabularnewline
 &  &  &  &  &  &  & 20 &  & 5236.93 & 262.6 &  & 5390.16 & - &  & 12 &  & 759 &  & 10.26\tabularnewline
 &  &  &  &  &  &  & 50 &  & 5140.32 & 298.9 &  & 5140.32 & - &  & 1 &  & 1292 &  & 0\tabularnewline
\hline 
\multirow{5}{*}{neos-2624317-amur} & \multirow{5}{*}{3.52} &  & \multirow{5}{*}{0} & \multirow{5}{*}{3.52} & \multirow{5}{*}{1.7} &  & 2 &  & 2.24 & 135.9 &  & 2.24 & 135.7 &  & 0 &  & 0 &  & 0\tabularnewline
 &  &  &  &  &  &  & 5 &  & 1.64 & 1407.7 &  & 1.64 & 1479.0 &  & 1 &  & 33 &  & 0.12\tabularnewline
 &  &  &  &  &  &  & 10 &  & 1.38 & 134.3 &  & 1.38 & 161.9 &  & 2 &  & 20 &  & 0.24\tabularnewline
 &  &  &  &  &  &  & 20 &  & 0.01 & 89.6 &  & 0.03 & 107.1 &  & 1 &  & 56 &  & 0.47\tabularnewline
 &  &  &  &  &  &  & 50 &  & 0 & 31.3 &  & 0 & 30.9 &  & 0 &  & 0 &  & 0\tabularnewline
\hline 
\multirow{5}{*}{aligninq} & \multirow{5}{*}{2713} &  & \multirow{5}{*}{2654.74} & \multirow{5}{*}{2713} & \multirow{5}{*}{3.0} &  & 2 &  & 2655.52 & 46.4 &  & 2655.52 & 42.5 &  & 1 &  & 8 &  & 0\tabularnewline
 &  &  &  &  &  &  & 3 &  & 2655.52 & 43.8 &  & 2655.52 & 40.0 &  & 1 &  & 10 &  & 0\tabularnewline
 &  &  &  &  &  &  & 6 &  & 2655.51 & 42.1 &  & 2655.52 & 45.3 &  & 3 &  & 57 &  & 0.02\tabularnewline
 &  &  &  &  &  &  & 11 &  & 2655.05 & 39.4 &  & 2655.52 & 81.6 &  & 2 &  & 129 &  & 0.8\tabularnewline
 &  &  &  &  &  &  & 30 &  & 2654.74 & 40.9 &  & 2655.52 & 398.9 &  & 3 &  & 548 &  & 1.34\tabularnewline
\hline 
\multirow{5}{*}{mcsched} & \multirow{5}{*}{211913} &  & \multirow{5}{*}{193774.75} & \multirow{5}{*}{211913} & \multirow{5}{*}{32.4} &  & 2 &  & 208779.90 & - &  & 208779.9 & - &  & 0 &  & 0 &  & 0\tabularnewline
 &  &  &  &  &  &  & 5 &  & 205575.23 & 430.8 &  & 205726.4 & - &  & 2 &  & 951 &  & 2.39\tabularnewline
 &  &  &  &  &  &  & 10 &  & 197814.36 & 727.2 &  & 197814.36 & - &  & 1 &  & 2868 &  & 0\tabularnewline
 &  &  &  &  &  &  & 20 &  & 193906.70 & 311.0 &  & 193906.7 & - &  & 1 &  & 6370 &  & 0\tabularnewline
 &  &  &  &  &  &  & 50 &  & 193775.74 & 278.3 &  & 193775.74 & - &  & 1 &  & 15460 &  & 0\tabularnewline
\hline 
\multirow{5}{*}{neos-911970} & \multirow{5}{*}{54.76} &  & \multirow{5}{*}{23.26} & \multirow{5}{*}{54.76} & \multirow{5}{*}{11.0} &  & 2 &  & 48.24 & - &  & 48.37 & - &  & 0 &  & 0 &  & 2.11\tabularnewline
 &  &  &  &  &  &  & 5 &  & 54.38 & - &  & 54.76 & - &  & 0 &  & 0 &  & 100\tabularnewline
 &  &  &  &  &  &  & 10 &  & 54.76 & 1098.8 &  & 54.76 & - &  & 1 &  & 709 &  & 0\tabularnewline
 &  &  &  &  &  &  & 20 &  & 52.40 & 278.5 &  & 54.76 & - &  & 3 &  & 1084 &  & 100\tabularnewline
 &  &  &  &  &  &  & 27 &  & 23.26 & 35.6 &  & 23.26 & 32.7 &  & 0 &  & 0 &  & 0\tabularnewline
\hline 
\multirow{5}{*}{neos-3627168-kasai} & \multirow{5}{*}{988585.62} &  & \multirow{5}{*}{945808.1} & \multirow{5}{*}{988585.62} & \multirow{5}{*}{859.4} &  & 2 &  & 981110.97 & - &  & 981110.97 & - &  & 0 &  & 0 &  & 0\tabularnewline
 &  &  &  &  &  &  & 5 &  & 976858.45 & 2720.6 &  & 976864.79 & \multicolumn{1}{c}{-} &  & 1 &  & 91 &  & 0.05\tabularnewline
 &  &  &  &  &  &  & 10 &  & 974695.38 & 280.6 &  & 974884.53 & 458.2 &  & 1 &  & 237 &  & 1.36\tabularnewline
 &  &  &  &  &  &  & 20 &  & 968889.48 & 106.8 &  & 968892.37 & 219.0 &  & 1 &  & 359 &  & 0.01\tabularnewline
 &  &  &  &  &  &  & 50 &  & 958310.99 & 71.9 &  & 958310.99 & 75.4 &  & 1 &  & 212 &  & 0\tabularnewline
\hline 
\multirow{5}{*}{neos-4650160-yukon} & \multirow{5}{*}{59.88} &  & \multirow{5}{*}{-22688.55} & \multirow{5}{*}{59.89} & \multirow{5}{*}{-} &  & 2 &  & 56.74 & - &  & 56.74 & - &  & 0 &  & 0 &  & 0\tabularnewline
 &  &  &  &  &  &  & 5 &  & 49.28 & - &  & 49.28 & - &  & 0 &  & 0 &  & 0\tabularnewline
 &  &  &  &  &  &  & 10 &  & 20.93 & 299.1 &  & 21.68 & - &  & 1 &  & 660 &  & 1.92\tabularnewline
 &  &  &  &  &  &  & 20 &  & 15.85 & 247.1 &  & 16.23 & - &  & 1 &  & 929 &  & 0.86\tabularnewline
 &  &  &  &  &  &  & 50 &  & -69.58 & 231.9 &  & -69.58 & 2346.3 &  & 2 &  & 1440 &  & 0\tabularnewline
\hline 
\caption{Extended MIPLIB results for instances where consistency cuts improve the lower bound for one of the 5 decompositions.} \label{MIPLIB3}
\end{longtable}
\end{singlespace}
\end{landscape}
\end{small}

\end{document}